\documentclass[leqno, 10pt]{amsart}
\usepackage{amsmath, amsthm, amssymb, latexsym}
 \newtheorem{definition}{Definition}[section]
 \newtheorem{theorem}[definition]{Theorem}
 \newtheorem{lemma}[definition]{Lemma}
 \newtheorem{proposition}[definition]{Proposition}
 \newtheorem{corollary}[definition]{Corollary}

 \newtheorem*{theorem*}{Theorem}
\newtheorem*{proposition*}{Proposition}
\newtheorem*{lemma*}{Lemma}

 \theoremstyle{remark}
 
 \newtheorem{remark}[definition]{Remark}

  \newtheorem*{acknowledgements}{Acknowledgements}
  \newtheorem*{notation}{Notation}

\newcommand{\op}[1]{\operatorname{#1}}

\newcommand{\acou}[2]{\ensuremath{\langle #1 , #2 \rangle}}

\newcommand{\tr}{\op{tr}}
\newcommand{\Tra}{\ensuremath{\op{Trace}}}

\newcommand{\TR}{\ensuremath{\op{TR}}}

\newcommand{\Res}{\ensuremath{\op{Res}}}

\newcommand{\C}{\ensuremath{\mathbb{C}}} 
 
\newcommand{\N}{\ensuremath{\mathbb{N}}} 
\newcommand{\R}{\ensuremath{\mathbb{R}}} 
\newcommand{\Z}{\ensuremath{\mathbb{Z}}} 
\newcommand{\CZ}{\ensuremath{\mathbb{C}\!\setminus\!\mathbb{Z}}} 

\newcommand{\Rn}{\ensuremath{\R^{n}}}
\newcommand{\Rno}{\R^{n}\!\setminus\! 0}
\newcommand{\URn}{U\times\R^{n}}
\newcommand{\URno}{U\times(\R^{n}\!\setminus\! 0)}

\newcommand{\Ca}[1]{\ensuremath{\mathcal{#1}}}

\newcommand{\cD}{\ensuremath{\mathcal{D}}}
\newcommand{\cE}{\Ca{E}}

\newcommand{\cK}{\ensuremath{\mathcal{K}}}

\newcommand{\cS}{\ensuremath{\mathcal{S}}}

\newcommand{\pdo}{\ensuremath{\Psi}} 
\newcommand{\pdoi}{\ensuremath{\Psi^{\op{int}}}} 
\newcommand{\pdoz}{\ensuremath{\Psi^{\Z}}} 
\newcommand{\pdozev}{\ensuremath{\Psi^{\Z}_{\text{ev}}}} 
\newcommand{\pdozodd}{\ensuremath{\Psi^{\Z}_{\text{odd}}}} 
\newcommand{\pdozo}{\ensuremath{\Psi^{\Z}_{\text{reg}}}} 
\newcommand{\pdozsing}{\ensuremath{\Psi^{\Z}_{\text{sing}}}} 
\newcommand{\pdocz}{\ensuremath{\Psi^{\CZ}}} 
\newcommand{\pdoc}{\ensuremath{\Psi_{\op{c}}}} 
\newcommand{\psido}{$\Psi$DO} 
\newcommand{\psidos}{$\Psi$DOs} 
\newcommand{\psinf}{\ensuremath{\Psi^{-\infty}}} 
\newcommand{\psinfc}{\ensuremath{\Psi^{-\infty}_{c}}}

\newcommand{\reg}{{\text{reg}}}
\newcommand{\ord}{{\op{ord}}}

\newcommand{\xiy}{{\xi\rightarrow y}}
\newcommand{\yxi}{{y\rightarrow\xi}}

\newcommand{\supp}{\op{supp}}

\newcommand{\End}{\ensuremath{\op{End}}}

\newcommand{\hotimes}{\hat\otimes}

\begin{document}
\title{Traces on Pseudodifferential Operators\\ and sums of commutators} 

\author{Rapha\"el Ponge}

 \address{Department of Mathematics, University of Toronto, 40 St George Street, Toronto, ON M5S 2E4, Canada.}
\email{ponge@math.utoronto.ca}
\thanks{Research partially supported by NSERC grant 341328-07 and by a New Staff Matching grant from the Connaught Fund of the University of Toronto}

\begin{abstract}
    The aim of this paper is to show that various known characterizations of traces on classical pseudodifferentials operators can actually be obtained by  
    very elementary considerations on pseudodifferential operators, using only basic properties of these operators. Thereby, we give a unified treatment 
    of the determinations of the space of traces (i) on \psidos\ of non-integer order or of regular parity-class, (ii) on integer order \psidos, (iii) on 
    \psidos\ of non-positive orders in dimension~$\geq 2$, and (iv) on \psidos\ of non-positive orders in dimension~$1$. 
\end{abstract}

\maketitle 
\numberwithin{equation}{section}

 \maketitle 
    
\section*{Introduction}
This paper deals with the description of traces and sum of commutators of classical pseudodifferential operators (\psidos) 
acting on the sections of a vector bundle $\cE$ over a compact manifold $M^{n}$. The results depend on the class of operators under consideration. 

First, if we consider integer order \psidos\ then an important result of Wodzicki~\cite{Wo:PhD} (see also~\cite{Gu:RTCAFIO}, \cite{Pa:NCRCTLSCP}) states 
that when $M$ is connected every trace is proportional to the noncommutative residue trace. The latter was discovered independently by 
Wodzicki~(\cite{Wo:LISA}, \cite{Wo:NCRF}) and Guillemin~\cite{Gu:NPWF}. Recall that the noncommutative residue of an integer order \psido\ 
is given by the integral of a density which in local coordinates can be expressed in terms of 
the symbol of degree $-n$ of the \psido. Alternatively, the noncommutative residue appears as the residual trace induced on integer 
\psidos\ by the analytic continuation of the usual trace to the class \psidos\ of noninteger complex orders. Since its discovery 
 it has found  numerous generalizations and applications (see, e.g., \cite{CM:LIFNCG}, \cite{FGLS:NRMB}, \cite{Gu:RTCAFIO}, \cite{Le:NCRPDOLPS}, 
 \cite{MMS:FIT}, \cite{PR:CDBFCF}, \cite{PR:TCCLG}, \cite{Po:JFA1}, \cite{Sc:NCRMCS}, \cite{Ug:CCGJMSOUWR}, \cite{Va:PhD}).

Following the terminology of~\cite{KV:GDEO} the analytic extension of the usual trace to noninteger order \psidos\ is called the \emph{canonical trace}. 
This is a trace in the sense that it vanishes on commutators $[P_{1}, P_{2}]$ such that $\ord P_{1}+\ord P_{2}$ is not an integer. Furthermore, it
makes sense on integer order \psidos\ such that their symbols have parity properties which make their noncommutative residue densities 
vanish. In this paper these \psidos\ are said to be of regular parity-class (see Section~\ref{sec:Schwartz-kernels} for the precise definition). It has been shown recently by 
Maniccia-Schrohe-Seiler~\cite{MSS:UKVT} and Paycha~\cite{Pa:NCRCTLSCP} that the tracial properties of the canonical trace characterize 
it among the linear forms on noninteger order \psidos\ and on regular parity-class \psidos.  

Next, when we consider the algebra of zero'th order \psidos\ we get other traces by composing any linear form on 
$C^{\infty}(S^{*}M)$ with the fiberwise trace of the zero'th order symbol of the \psido. Such traces are called \emph{leading symbol traces}. It has been shown by 
Wodzicki~\cite{Wo:PhD} that when $M$ is connected and has dimension~$\geq 2$ every trace on zero'th order \psidos\ is the sum of a leading 
symbol trace and of a constant multiple of the noncommutative residue. This result was rediscovered by Lescure-Paycha~\cite{LP:UMDEPDO} via the 
computation of the Hochschild homology of the algebra of zero'th order \psidos\ (which, at least in the continuous case, can also be found in~\cite{Wo:RCHS}).

Notice that in~\cite{LP:UMDEPDO} there is no distinction between the cases $n\geq 2$ and \mbox{$n=1$}. 
However, as noticed by Wodzicki~\cite{Wo:PhD}, as well as by the author, there is a specificity to the one dimensional case since in dimension 1
we get other traces beside the sums of leading symbol traces and constant multiples of the noncommutative residue (see below). 

The aim of this paper is to show that the aforementioned characterizations of traces on \psido\ algebras can all be obtained from elementary 
considerations on \psidos, using only very basic properties of these operators. Furthermore, this includes a characterization of the traces on zero'th order 
\psidos\ in dimension 1. 

In his Steklov Institute thesis~\cite{Wo:PhD} Wodzicki determined all the traces on integer order \psidos\ and on zero'th order \psidos, both in 
dimension~$\geq 2$ and in dimension~1. Unfortunately, the  proofs of Wodzicki did not appear elsewhere, so it is difficult to have access to them. 
According to Wodzicki~\cite{Wo:PC} the proofs follow from the determination of the commutator
spaces $\{\mathcal{P}_j,\mathcal{P}_k\}$, where $\mathcal{P}_j$ denotes the space of functions on $T^{*}M\setminus 0$ that are homogeneous of degree $j$. 

The approach of this paper differs  from that of~\cite{Wo:PhD} and can be briefly described as follows. 

The uniqueness of the canonical trace is an immediate 
consequence of the fact that any non-integer order (resp.~parity class) \psido\  is a sum of commutators of 
functions with non-integer order  (resp.~parity class) \psidos\  up to a smoothing operator 
(see Proposition~\ref{prop:Traces.sum-commutators-M1}). 

The uniqueness of the noncommutative residue follows from the fact that an integer order \psido\ supported on a local chart
is a sum of commutators of compactly supported \psidos\ of given specific types, modulo a constant multiple of a fixed given \psido\ with 
non-vanishing noncommutative residue (see Proposition~\ref{prop:TracesZ.commutators-U}). 

A difference with the approach of~\cite{Wo:PhD} is that we work with \psidos\ supported on a local chart, rather than with symbols defined 
on the whole cosphere bundle. Thus, our arguments are very much related to 
that~\cite{FGLS:NRMB} and~\cite{MSS:UKVT}, except that we make use of the  characterization of the \psidos\ in terms of their Schwartz kernels 
and of the interpretation due to~\cite{CM:LIFNCG} of the noncommutative residue of a \psido\ in terms of the logarithmic singularity of its Schwartz kernel 
near the diagonal. 
This leads us  to a simple argument showing that any smoothing operator on $\Rn$ can be written as a sum of commutators of coordinate 
functions with \psidos\ of order $-n+1$ (see Lemma~\ref{lem:TracesZ.smoothing-operators}). In particular, this allows us to prove the uniqueness of the 
noncommutative directly for \psidos, rather than for symbols (compare~\cite{FGLS:NRMB}). 

To deal with traces on zero'th order \psidos\ we observe that to a large extent in dimension~$\geq 2$ 
the sums of \psido\ commutators involved in the proof of the uniqueness of the 
noncommutative residue can be replaced by sum of commutators involving \psidos\ of order~$\leq 0$ 
(see  Proposition~\ref{prop:Traces0.commutators-U} for the precise statement). This allows us to characterize traces on zero'th order \psidos\ in 
dimension~$\geq 2$. In particular, this provides us with an alternative to the spectral sequence arguments of~\cite{LP:UMDEPDO}. 

The main ingredient in the determination of traces on zero'th order \psidos\ in dimension 1 is the observation that in dimension 1 
the symbol of degree $-1$ of a zero'th order \psido\ makes sense intrinsically as a section over the cosphere bundle 
$S^{*}M$ (Proposition~\ref{prop:Traces01.symbol-1}).
As a consequence we can define \emph{subleading symbol traces} in the same way as leading symbol traces are defined. The noncommutative residue is an 
example of subleading symbol trace and we show that in dimension 1 any trace on zero'th order \psidos\ can be uniquely written as 
the sum of a leading symbol trace and of a subleading symbol trace (Theorem~\ref{thm:Traces01.main}). 
An interesting consequence is that for scalar \psidos\ the commutator space of zero'th order \psidos\ agrees with the space of 
\psidos\ of order~$\leq -2$.   

 This paper is organized as follows. In Section~\ref{sec:Schwartz-kernels}, we recall some basic facts on  \psidos\ and their 
Schwartz kernels. In Section~\ref{sec:NCRCT}, we collect some key definitions and properties of the noncommutative residue and of the canonical trace. 
In Section~\ref{sec:Traces-TR}, we show the uniqueness of the canonical trace. In Section~\ref{sec:TracesZ}, we prove that of the noncommutative residue.
In Section~\ref{sec:TracesZ0}, we characterize traces on zero'th order \psidos\ in dimension~$\geq 2$. In Section~\ref{sec:TracesZ01}, we deal with the one dimensional case.

\begin{notation}
     Throughout the paper we let $M^{n}$ denote a compact manifold of dimension $n$ and we let $\cE$ denote a vector bundle of rank $r$ over $M$.
 \end{notation}

\begin{acknowledgements}
I am very grateful to Sylvie Paycha and Mariusz Wodzicki  for stimulating discussions related to the subject matter of this paper, and to Maciej Zworksi 
for his pointer to the reference~\cite{Ka:OCCHTZ}.  I also wish to thank for its hospitality the University of California at Berkeley
where this paper was written. 
\end{acknowledgements}

\section{Pseudodifferential operators and their Schwartz kernels}
\label{sec:Schwartz-kernels}
In this section we recall some notation and results about (classical) \psidos\ and their Schwartz  kernels. 

Let $U$ be an open subset of $\Rn$. The symbols on $\URn$ are defined as follows. 

\begin{definition}
    1) $S_{m}(\URn)$, $m\in \C$, consists of smooth functions $p(x,\xi)$ on $\URno$ such that  $p(x,\lambda\xi)=\lambda^{m}p(x,\xi)$ for any 
    $\lambda>0$.\smallskip
    
    2) $S^{m}(\URn)$, $m\in \C$, consists of smooth functions $p(x,\xi)$ on  $\URn$ admitting an 
asymptotic expansion $p(x,\xi)\sim \sum_{j\geq 0} p_{m-j}(x,\xi)$, $p_{m-j}\in S_{m-j}(\URn)$, in the sense that, for any integer $N$ and 
any compact $K \subset U$, we have estimates:  
    \begin{equation}
        |\partial_{x}^{\alpha}\partial_{\xi}^{\beta}(p-\sum_{j < N}p_{m-j})(x,\xi)|\leq C_{NK\alpha\beta}|\xi|^{\Re m -N-|\beta|}, \quad x\in K, \ 
        |\xi|\geq 1.
         \label{eq:PsiDO.asymptotic-expansion-symbol}
    \end{equation}
\end{definition}

Given a symbol $p\in S^{m}(\URn)$ we let $p(x,D)$ denote the operator from $C^{\infty}_{c}(U)$ to 
$C^{\infty}(U)$ given by
\begin{equation}
    p(x,D)u(x)=(2\pi)^{-n}\int e^{ix.\xi}p(x,\xi)\hat{u}(\xi)d\xi, \quad u \in C^{\infty}_{c}(U).
\end{equation}
A \psido\ of order $m$ on $U$ is an operator $P$ from  $C^{\infty}_{c}(U)$ to $C^{\infty}(U)$ of the form
\begin{equation}
     P=p(x,D)+R,
\end{equation}
 with $p\in S^{m}(\URn)$ and $R$ a smoothing operator (i.e.~$R$ is given by a smooth Schwartz kernel).  
 We let $\Psi^{m}(U)$ denote the space of  \psidos\ of order $m$ on $U$. 
 
 Any \psido\ on $U$ is a continuous operator  from $C^{\infty}_{c}(U)$ to $C^{\infty}(U)$ and its
 Schwartz kernel is smooth off the diagonal. We then define \psidos\ on $M$ acting on sections of $\cE$ as follows. 
 
 \begin{definition}
     $\Psi^{m}(M,\cE)$, $m\in \C$, consists of continuous operators $P$ from $C^{\infty}_{c}(M,\cE)$ to $C^{\infty}(M,\cE)$ such that the Schwartz kernel of 
$P$ is smooth off the diagonal and in any open of trivializing local coordinates $U \subset \Rn$ we can write $P$ in the form
    \begin{equation}
        P=p(x,D)+R,
         \label{eq:PsiDO.PsiDO-chart}
    \end{equation}
for some symbol $p\in S^{m}(\URn)\otimes \End \C^{r}$ and some smoothing operator $R$. 
 \end{definition}
 
In addition, we let $\Psi^{-\infty}(M,\cE)$ denote the space of smoothing operators on $M$ acting on sections of $\cE$. 

Let us now recall the description of \psidos\ in terms of their Schwartz kernels. This description is well-known to experts (see, e.g., 
\cite{BG:CHM}, \cite{Ho:ALPDO3}, \cite{Me:APSIT}, \cite{Po:NCR}). The exposition here follows that of \cite{BG:CHM} and~\cite{Po:NCR}.

First, for $\tau$ in $\cS'(\Rn)$ and $\lambda \in \R\setminus 0$ we let $\tau_{\lambda}\in \cS'(\Rn)$ be defined by
     \begin{equation}
           \acou{\tau_{\lambda}(\xi)}{u(\xi)}=|\lambda|^{-n} \acou{\tau(\xi)}{u(\lambda^{-1}\xi)} \qquad \forall u\in\cS(\Rn). 
       \end{equation}
In the sequel it will be convenient to also use the notation $\tau(\lambda \xi)$ to denote $\tau_{\lambda}$.  In any case, we say that  $\tau$ 
is \emph{homogeneous} of degree $m$, $m\in\C$, when $\tau_{\lambda}=\lambda^m \tau$ for any $\lambda>0$. 

It is natural to ask whether a homogeneous functions on $\Rno$ could be extended into a homogeneous distribution on $\Rn$. This problem is completely 
solved by:

\begin{lemma}[{\cite[Thm.~3.2.3, Thm.~3.2.4]{Ho:ALPDO1}}]\label{lem:Homogenenous-extension}
    Let $p(\xi) \in C^{\infty}(\Rno)$ be homogeneous of degree $m$, $m \in \C$.\smallskip
    
    1) If $m$ is not an integer~$\leq -n$, then $p(\xi)$ can be uniquely extended into a homogeneous distribution $\tau(\xi)$ in $\cS'(\Rn)$.\smallskip
    
    2) If $m$ is an integer $\leq -n$,  then at best we can extend $p(\xi)$ into a distribution $\tau(\xi)$ in $\cS'(\Rn)$ such that, for any $\lambda 
       >0$, we have
\begin{equation}
      \tau(\lambda\xi)=\lambda^{m}\tau(\xi) +\lambda^{m}\log \lambda\sum_{|\alpha|=-(m+n)} c_{\alpha}(p)\delta^{(\alpha)},
     \label{eq:NCR.log-homogeneity}
\end{equation}
  where we have let $c_{\alpha}(p) = \int_{S^{n-1}}\frac{(-\xi)^{\alpha}}{\alpha!} p(\xi)d^{n-1}\xi$. In particular $p(\xi)$ admits a homogeneous 
  extension if and only if all the coefficients $c_{\alpha}(p)$ vanish. 
\end{lemma}

In the sequel for any $\tau\in \cS'(\Rn)$ we let $\check{\tau}$ denote its inverse Fourier transform. Let $\lambda>0$. For any $u \in \cS(\Rn)$ we 
have 
\begin{equation}
    \acou{(\check{\tau})_{\lambda}}{u}=|\lambda|^{-n}\acou{\tau}{(u_{\lambda^{-1}})^{\vee}}=\acou{\tau}{(\check{u})_{\lambda}}= 
    |\lambda|^{-n}\acou{(\tau_{\lambda^{-1}})^{\vee}}{u},  
\end{equation}
that is, we have $\check{\tau}(\lambda \xi)=|\lambda|^{-n}(\tau(\lambda^{-1} \xi))^{\vee}$. From this we deduce that:\smallskip

- $\tau$ is homogeneous of degree $m$ if and only if $\check{\tau}$ is homogeneous of degree $-(m+n)$;\smallskip

- $\tau$ satisfies~(\ref{eq:NCR.log-homogeneity}) if and only if we have
\begin{equation}
    \check{\tau}(\lambda.y)= \lambda^{\hat{m}} \check{\tau}(y) - \lambda^{\hat{m}}\log \lambda 
    \sum_{|\alpha| =\hat{m}} (2\pi)^{-n}c_{\alpha}(p) (-iy)^{\alpha} \qquad  \forall \lambda \in \R\setminus 0.
     \label{eq:NCR.log-homogeneity-kernel}
\end{equation}

In the sequel we set $\N_{0}=\N\cup\{0\}$ and we let $\cS'_{\reg}(\Rn)$  denote the space of tempered distributions on $\Rn$ which are smooth outside the 
origin equipped with the locally convex topology induced by that of $\cS'(\Rn)$ and $C^{\infty}(\Rno)$. 
  
\begin{definition}
 The space $\cK_{m}(\URn)$, $m\in\C$, consists of distributions $K(x,y)$ in $C^\infty(U)\hotimes\cS'_{\reg}(\Rn)$ such that, for any $\lambda>0$, 
 we have
\begin{equation}
      K(x,\lambda y)= \left\{
      \begin{array}{ll}
          \lambda^m K(x,y) & \text{if $m\not \in \N_{0}$},  \\
          \lambda^m K(x,y) + \lambda^m\log\lambda
                    \sum_{|\alpha|=m}c_{K,\alpha}(x)y^\alpha & \text{if $m\in \N_{0}$},
      \end{array}\right.
    \label{eq:NCR.log-homogeneity-Km}
\end{equation} 
where the functions $c_{K,\alpha}(x)$, $|\alpha|=m$, are in $C^{\infty}(U)$ when $m\in \N_{0}$.
\end{definition}

\begin{definition}\label{def:NCR.cKm2}
$\cK^{m}(\URn)$, $m\in \C$, consists of distributions $K$ in $\cD'(\URn)$ with an asymptotic expansion 
     $K\sim \sum_{j\geq0}K_{m+j}$,  $K_{l}\in \cK_{l}(\URn)$, 
 in the sense that,  for any integer $N$, provided $J$ is large enough we have
\begin{equation}
    K-\sum_{j\leq J}K_{m+j}\in  C^{N}(\URn).
    \label{eq:NCR.cKm2}
\end{equation}
\end{definition}

Using Lemma~\ref{lem:Homogenenous-extension} and the discussion that follows we get:
\begin{lemma}\label{lem:NCR.extension-symbolU}
 1) If $p(x,\xi)\in S_{m}(\URn)$ then $p(x,\xi)$ can be extended into a distribution $\tau(x,\xi)\in C^{\infty}(U)\hotimes 
 \cS_{\reg}'(\Rn)$ such that $K(x,y):=\check{\tau}_{\xiy}(x,y)$ belongs to $\cK_{\hat{m}}(\URn)$, $\hat{m}=-(m+n)$. Furthermore,  when $m$ is an 
 integer~$\leq -n$ with the notation 
 of~(\ref{eq:NCR.log-homogeneity-Km}) we have $c_{K,\alpha}(x)= (2\pi)^{-n}\int_{S^{n-1}}\frac{(i\xi)^{\alpha}}{\alpha!}p(x,\xi)d^{n-1}\xi$.\smallskip
 
 2) If $K(x,y)\in \cK_{\hat{m}}(\URn)$ then the restriction of $\hat{K}_{\yxi}(x,\xi)$ to $\URno$ belongs to $S_{m}(\URn)$.
\end{lemma}

This lemma is a key ingredient in the characterization of \psidos\ below.
\begin{proposition}\label{prop:LS.characterisation-kernel}
 Let $P:C_{c}^\infty(U)\rightarrow C^\infty(U)$ be a continuous operator with Schwartz kernel $k_{P}(x,y)$. 
 Then the following are equivalent:\smallskip
 
 (i) $P$ is a \psido\ of of order $m$, $m\in \C$.\smallskip
 
  (ii) We can write $k_{P}(x,y)$ in the form
  \begin{equation}
     k_{P}(x,y)=K(x,x-y) +R(x,y) ,
      \label{eq:PsiHDO.characterization-kernel}
 \end{equation}
with  $K\in \cK^{\hat{m}}(\URn)$, $\hat{m}=-(m+n)$, and $R \in C^{\infty}(U\times U)$.
 
   Moreover, if (i) and (ii) hold and if in the sense of~(\ref{eq:NCR.cKm2}) we have $K\sim \sum_{j\geq 0}K_{\hat{m}+j}$, 
  $K_{l}\in\cK_{l}(\URn)$,  then 
  $P$ has symbol $p\sim \sum_{j\geq 0} p_{m-j}$, $p_{l}\in S_{l}(\URn)$, where 
  $p_{m-j}(x,\xi)$ is the restriction to $\URno$ of $(K_{m+j})^{\wedge}_{\yxi}(x,\xi)$. 
\end{proposition}

The above description of \psidos\ allows us to determine the singularities near the diagonal of the Schwartz kernels of \psidos. In particular, we 
have: 

\begin{proposition}
    Let $P:C^{\infty}(M,\cE)\rightarrow C^{\infty}(M,\cE)$ be a \psido\ of integer order $m$. Then in local coordinates its 
    Schwartz kernel $k_{P}(x,y)$ has a behavior near 
the diagonal $y=x$ of the form
\begin{equation}
        k_{P}(x,y)=\sum_{-(m+n)\leq j \leq -1}a_{j}(x,x-y) -c_{P}(x) \log |y-x| +\op{O}(1),
         \label{eq:Log.behavior-kP}
\end{equation}
where $a_{j}(x,y)\in C^{\infty}(\URno)$ is homogeneous of degree $j$ with respect to $y$ and $c_{P}(x)\in C^{\infty}(U)$ is 
 given by 
\begin{equation}
    c_{P}(x)=(2\pi)^{-n}\int_{S^{n-1}}p_{-n}(x,\xi)d^{n-1}\xi. 
     \label{eq:NCR.formula-cP}
\end{equation}
\end{proposition}

The description~(\ref{eq:Log.behavior-kP}) of the behavior of $k_{P}(x,y)$ depends on the choice of the local coordinates, but the coefficient 
$c_{P}(x)$ makes sense intrinsically, for we have: 

\begin{proposition}[\cite{CM:LIFNCG}]\label{prop:Log-sing.density}
    The coefficient $c_{P}(x)$ in~(\ref{eq:Log.behavior-kP}) makes sense globally on $M$ as an $\End \cE$-valued 1-density. 
\end{proposition}

The point is that if $\phi : U'\rightarrow U$ is a change of local coordinates then we have 
\begin{equation}
    c_{\phi^{*}P}(x)=|\phi'(x)|c_{P}(\phi(x)) \qquad \forall P\in \Psi^{m}(U),
\end{equation}
which shows that $c_{P}(x)$ behaves like a 1-density (detailed proofs of this result can be found in \cite{GVF:ENCG}  and~\cite{Po:NCR}). 

Finally, we recall some definitions and properties of parity-class \psidos. 

\begin{definition}
1) We say that $P\in \Psi^{m}(M,\cE)$, $m \in \Z$, is odd-class 
if in local coordinates its symbol $p\sim 
    \sum_{j\geq 0}p_{m-j}$ is so that  $p_{m-j}(x,-\xi)=(-1)^{m-j}p_{m-j}(x,\xi)$ for all $j\geq 0$.\smallskip
    
2)     We say that $P\in \Psi^{m}(M,\cE)$, $m \in \Z$, is even-class 
if in local coordinates its symbol $p\sim 
    \sum_{j\geq 0}p_{m-j}$ is such that  $p_{m-j}(x,-\xi)=(-1)^{m-j+1}p_{m-j}(x,\xi)$  $\forall j\geq 0$.
\end{definition}

We let $\pdozodd(M,\cE)$ (resp.~$\pdozev(M,\cE)$) denote the space of odd-class (resp.~even-class) \psidos. Any differential 
operator  is odd-class and any parametrix of an odd-class (resp.~even-class) elliptic \psido\ is again odd-class \psido\ (resp.~even-class \psido). 
Furthermore, if $P$ and $Q$ are in $ \pdozodd(M,\cE)\cup \pdozev(M,\cE)$, then the operator $PQ$ is an odd-class \psido~(resp.~even-class \psido) if 
the parity classes of $P$ and $Q$  agree (resp.~don't agree). 
    
\begin{definition}
1)   We say that $P\in \Psi^{m}(M,\cE)$, $m \in \Z$, is of regular parity-class if $P$ is odd-class and $n$ is odd or if $P$ is even-class and  $n$ is 
   even.\smallskip
   
2) We that $P\in \Psi^{m}(M,\cE)$, $m \in \Z$, is of singular parity-class if $P$ is odd-class and $n$ is even, or if $P$ is even-class and  $n$ is odd.
\end{definition}

We let $\pdozo(M,\cE)$ (resp.~$\pdozsing(M,\cE)$) denote the space of \psidos\ of regular  (resp.~singular) parity-class. Notice that the product of an even-class 
\psido\ and of a regular parity-class \psido\ is a regular parity-class \psido, as is the product of an even-class \psido\ and of a singular parity-class \psido. Furthermore, 
we have:

\begin{proposition}\label{eq:Even-odd,vanishing-cP}
    If $P\in \pdozo(M,\cE)$ then the density $c_{P}(x)$ vanishes. 
\end{proposition}

\section{The Noncommutative residue and the canonical trace}
\label{sec:NCRCT}
In this section we recall the main definitions and properties of the noncommutative residue and of the canonical trace. (In addition to the original 
references \cite{Gu:NPWF}--\cite{Gu:RTCAFIO}, \cite{KV:GDEO} and \cite{Wo:LISA}--\cite{Wo:NCRF}, we refer to~\cite{Po:NCR} for 
a detailed exposition  along the lines below.)


Let $\pdoi(M,\cE) = \cup_{\Re m < -n}\pdo^{m}(M,\cE)$ be the class of \psidos\ whose symbols are integrable with respect to the $\xi$-variable. 
If $P\in \pdoi(M,\cE)$ then 
the restriction  of its Schwartz kernel $k_{P}(x,y)$ to the diagonal of $M\times M$ defines a smooth $\End\cE$-valued 1-density $k_{P}(x,x)$.
As by assumption $M$ is compact we see that 
$P$ is a trace-class operator and we have
\begin{equation}
    \Tra (P) = \int_{M} \tr_{\cE}k_{P}(x,x).
\end{equation}

In fact, the map $P\rightarrow k_{P}(x,x)$  can be analytically extended to a map $P\rightarrow t_{P}(x)$ defined on the class $\pdocz(M,\cE)$ of 
non-integer order \psidos\ and on regular parity-class \psidos. More precisely, if $(P(z))_{z\in \C}$ is a holomorphic family of \psidos\ of 
non-integer orders as in~\cite{Gu:RTCAFIO} and~\cite{KV:GDEO}, then $(t_{P}(z))_{z\in \C}$ is a holomorphic family of $\End\cE$-values densities. 

Let $P\in \pdoz(M,\cE)$ and let $(P(z))_{z\in \C}$ be a family of \psidos\ such that $P(0)=P$ and $\ord P(z)=z+\ord P$. Then the map $z\rightarrow 
t_{P(z)}(x)$ has at worst a simple pole singularity near $z=0$ and we have
\begin{equation}
    \Res_{z=0}t_{P(z)}(x)=-c_{P}(x),
     \label{eq:NCRCT.residue-formula-tP}
\end{equation}
where $c_{P}(x)$ is the density defined by the logarithmic singularity of the Schwartz kernel of $P$ (cf.~Proposition~\ref{prop:Log-sing.density}). 
Furthermore, if $P$ is of regular parity-class then, under suitable conditions on the family $(P(z))$ (see, e.g.,~\cite{Pa:NCRCTLSCP}, \cite{Po:NCR}), we have
\begin{equation}
    \lim_{z\rightarrow 0}t_{P(z)}(x)=t_{P}(x).
\end{equation}

The \emph{canonical trace} is the functional on $\pdocz(M,\cE)\cup \pdozo(M,\cE)$ such that
\begin{equation}
    \TR P= \int_{M}\tr_{\cE}t_{P}(x) \qquad \forall P\in \pdocz(M,\cE)\cup \pdozo(M,\cE).
\end{equation}
This is the unique analytic continuation to $\pdocz(M,\cE)\cup \pdozo(M,\cE)$ of the functional $\Tra$. Moreover, the following holds.

\begin{proposition}\label{thm:NCR.TR.global1}
1)  We have $\TR [P_{1},P_{2}]=0$  whenever $\ord P_{1}+\ord P_{2}\not\in\Z$.\smallskip
 
 2) $\TR$ vanishes on $[\pdozodd(M,\cE), \pdozo(M,\cE)]$ and $[\pdozev(M,\cE),\pdozsing(M,\cE)]$.
\end{proposition}

Next, the noncommutative residue of an operator $P \in \pdoz(M,\cE)$ is
\begin{equation}
     \Res P =\int_{M}\tr_{\cE} c_{P}(x).
\end{equation}
Because of~(\ref{eq:NCR.formula-cP}) we recover the usual definition of the noncommutative residue. Moreover, if $(P(z))_{z\in \C}$ is a holomorphic family 
of \psidos\ such that $P(0)=P$ and $\ord P(z)=z+\ord P$, then by~(\ref{eq:NCRCT.residue-formula-tP}) we have
\begin{equation}
    \Res_{z=0} \TR P(z)= - \Res P
     \label{eq:NCRCT.residue-formula-global}
\end{equation}

Notice that the noncommutative residue vanishes on \psidos\ of integer order~$\leq -(n+1)$, including smoothing operators, and it also vanishes on \psidos\ 
of regular parity-class. In addition, using~(\ref{eq:NCRCT.residue-formula-global}) we get: 

\begin{proposition}\label{prop:NCRCT.trace-NCR} 
 The noncommutative residue is a trace on the algebra $\pdoz(M,\cE)$. 
\end{proposition}

Next, the trace properties of the canonical trace and of the noncommutative residue mentioned in Proposition~\ref{thm:NCR.TR.global1} and 
Proposition~\ref{prop:NCRCT.trace-NCR}  characterize these functionals. First, we have: 

\begin{theorem}[\cite{Wo:PhD}, \cite{Gu:RTCAFIO}]\label{thm:NCRCT.unicity-NCR} 
 If $M$ connected then any trace on $\pdoz(M,\cE)$ is a constant multiple of the noncommutative residue. 
\end{theorem}

Concerning the canonical trace the following holds.  

\begin{theorem}[\cite{MSS:UKVT}, \cite{Pa:NCRCTLSCP}]\label{thm:NCRCT.unicity-TR}
    1) Any linear map $\tau:\pdocz(M,\cE)\rightarrow \C$ vanishing on $[C^{\infty}(M), \pdocz(M,\cE)]$ and $[\psinf(M,\cE),\psinf(M,\cE)]$ 
    is a constant multiple of the canonical trace.\smallskip
    
    2) Any linear map $\tau:\pdozo(M,\cE)\rightarrow \C$ vanishing on $[C^{\infty}(M), \pdozo(M,\cE)]$ and $[\psinf(M,\cE),\psinf(M,\cE)]$ 
    is a constant multiple of the canonical trace.
\end{theorem}

In addition to the noncommutative residue, on zero'th order \psidos\ we have many other traces. For an operator $P\in \Psi^{0}(M,\cE)$ the zero'th 
order symbol uniquely defines a section $\sigma_{0}(P)\in C^{\infty}(S^{*}M,\End \cE)$ 
(for the sake of brevity we also denote by $\cE$ its pushforward by the canonical projection of $S^{*}M$ onto $M$). Then \emph{any} linear form $L$ 
on $C^{\infty}(S^{*}M)$ gives rise to a trace $\tau_{L}$ on $\Psi^{0}(M,\cE)$ defined by the formula 
\begin{equation}
    \tau_{L}(P):=L[\tr_{\cE}\sigma_{0}(P)] \qquad \forall P\in \Psi^{0}(M,\cE).
\end{equation}
Such a trace is called a \emph{leading symbol trace}. 


\begin{theorem}[\cite{Wo:PhD}, \cite{LP:UMDEPDO}]\label{thm:NCRCT.Traces0}
    Suppose that $M$ is connected and has dimension $\geq 2$. Then any trace on $\Psi^{0}(M,\cE)$ can be uniquely written as the sum of a leading symbol trace 
    and of a constant multiple of the noncommutative residue. 
\end{theorem}

We refer to Section~\ref{sec:TracesZ01} for the description of the traces on $\Psi^{0}(M,\cE)$ when $n=1$.

\section{Uniqueness of the canonical trace}
\label{sec:Traces-TR}
In this section we shall give a proof of Theorem~\ref{thm:NCRCT.unicity-TR} about the uniqueness of the canonical trace. 
First, we have: 

\begin{lemma}\label{lem:Traces.criterion-logarithmic-kernel}
    Let  $K(x,y)\in \cK_{0}(\Rn\times \Rn)$ be homogeneous of degree $0$ with respect to $y$ and let $P\in \Psi^{-n}(\Rn)$ be the \psido\ with 
    Schwartz kernel $k_{P}(x,y)=K(x,x-y)$. Then $P$ can be written in the form
 \begin{equation}
    P=  [ x_{0}, P_{0}]+\ldots+ [x_{n}, P_{n}], 
     \label{eq:Traces.sum-commutators-K0}
\end{equation}    
   with $P_{1}, \ldots, P_{n}$ in $\Psi^{-n+1}(\Rn)$. Moreover, if $K_{0}(x,-y)=-K_{0}(x,y)$ then 
   $P_{1},\ldots,P_{n}$  can be chosen to be of regular parity-class.
\end{lemma}
\begin{proof}
  For $j=1,\ldots, n$ and for $(x,y)\in \Rn\times \Rno$ set $K^{(j)}(x,y)= y_{j}|y|^{-2}K(x,y)$. 
 As $K^{(j)}(x,y)$ is smooth for $y\neq 0$ and is homogeneous with respect to $y$ of 
 degree~$-1$ we see that $K^{(j)}(x,y)$ is an element of $\cK_{-1}(\R \times \R)$. Therefore, 
 by Proposition~\ref{prop:LS.characterisation-kernel} the operator $P_{j}$ with 
 kernel $k_{P_{j}}(x,y)=K^{(j)}(x,x-y)$ is a \psido\ of order $-n+1$. 
 Moreover, observe that 
 the  Schwartz kernel of  $\sum_{j=1}^{n}[x_{j}, P_{j}]$ is
  \begin{equation}
   \sum_{1\leq j \leq n} (x_{j}-y_{j})^{2}|x-y|^{-2}K(x,x-y)=K(x,x-y)=k_{P}(x,y).
  \end{equation}
Hence $P=  [ x_{0}, P_{0}]+\ldots+ [x_{n}, P_{n}]$.

Assume now that $K(x,-y)=-K(x,y)$. Then we have $K^{(j)}(x,-y)=K^{(j)}(x,y)$. By Proposition~\ref{prop:LS.characterisation-kernel} 
the symbol of $P_{j}$ is $p^{(j)}(x,\xi)\sim 
p_{-n+1}^{(j)}(x,\xi)$ with $p_{-n+1}^{(j)}(x,\xi)=(K^{(j)})^{\wedge}_{\yxi}(x,\xi)$, so we have
\begin{equation}
  p_{-n+1}^{(j)}(x,-\xi)=  (K^{(j)}(x,-y))^{\wedge}_{\yxi}(x,\xi)=(K^{(j)})^{\wedge}_{\yxi}(x,\xi)=p_{-n+1}(x,\xi). 
\end{equation}
Since $1=(-1)^{-n+1}$ when $n$ is odd and $1=(-1)^{-n+1+1}$ when $n$ is even, this shows that $P_{j}$ is odd-class when $n$ is odd and is even-class when 
$n$ is even. In any case $P_{j}$ is  of regular parity-class. The lemma is thus proved.
\end{proof}

\begin{lemma}[\cite{FGLS:NRMB}, \cite{Gu:GLD}]\label{lem:Traces.sum-commutators-Rn1}
    Let $P\in \pdo^{m}(\Rn)$, $m\in \C$, and assume that either $m$ is not an integer~$\geq -n$, or we have $c_{P}(x)=0$. 
    Then $P$ can be written in the form
\begin{equation}
    P=  [x_{1}, P_{1}]+\ldots+ [x_{n}, P_{n}] +R,
     \label{eq:Traces.sum-commutators-cP0}
\end{equation}    
   with $P_{1}, \ldots, P_{n}$ in $\Psi^{m+1}(\Rn)$ and $R\in \psinf(\R^{n})$. Furthermore, if $P$ is in $\pdozo(\Rn)$ then 
   $P_{1},\ldots,P_{n}$ can be chosen to be in $\pdozo(\Rn)$ as well. 
\end{lemma}
\begin{proof}
Let us first assume that either $m$ is not an integer~$\geq -n$ or the symbol of degree $-n$ of $P$ is zero. Then we can put $P$ in the form~(\ref{eq:Traces.sum-commutators-cP0}) 
as follows. Let $p(x,\xi)\sim \sum_{j\geq 0}p_{m-j}(x,\xi)$ 
be the symbol of $P$. Then the Euler identity tells us that 
$\sum_{k=1}^{n}\xi_{k}\partial_{\xi_{k}}p_{m-j}=(m-j)p_{m-j}$, so we have 
\begin{equation}
     \sum_{k=1}^{n}\partial_{\xi_{k}}[\xi_{k}p_{m-j}]= n 
     p_{m-j}+\sum_{k=1}^{n}\xi_{k}\partial_{\xi_{k}}p_{m-j}=(m-j+n)p_{m-j}.
     \label{eq:TracesZ.Euler-formula}
\end{equation}

By assumption either $m$ is not an integer or $p_{-n}(x,\xi)=0$, so for $k=1,\ldots,n$ there always exists $P_{k} \in \pdo^{m+1}(\Rn)$ with 
symbol $p^{(k)}\sim \frac{1}{i}\sum_{j \geq 0}\frac{1}{m-j+n}\xi_{k}p_{m-j}$. Then thanks to~(\ref{eq:TracesZ.Euler-formula})
the operator $\sum_{k=1}^{n}[x_{k},P_{k}]$ has symbol
\begin{equation}
q=\sum_{k=1}^{n}\partial_{\xi_{k}}p^{(k)}
\sim   \sum_{j \geq 0} \frac{1}{m-j+n}\sum_{k=1}^{n}\partial_{\xi_{k}}[\xi_{k}p_{m-j}] \sim \sum_{j\geq 0} p_{m-j} \sim p.
\end{equation}
This shows that $\sum_{k=1}^{n}[x_{k},P_{k}]$ has same symbol as $P$, so it agrees with $P$ up to a smoothing operator. Furthermore, if $P$ is 
of regular parity-class, then each operator $P_{k}$ is a \psido\ of regular parity-class too. 

It remains to deal with the case where $m$ is an integer~$\geq -n$ and $c_{P}(x)=0$. Let $p(x,\xi)\sim \sum_{j\geq 0}p_{m-j}(x,\xi)$ be the symbol of $P$. 
By~(\ref{eq:NCR.formula-cP}) for any $x\in \Rn$ we have 
\begin{equation}
    \int_{S^{n-1}}p_{-n}(x,\xi)d^{n-1}\xi=(2\pi)^{n}c_{P}(x)=0.
\end{equation}
Therefore, by 
Lemma~\ref{lem:NCR.extension-symbolU}  we can extend $p_{-n}(x,\xi)$ into a distribution $\tau(x,\xi)$ in $C^{\infty}(U)\hotimes 
 \cS_{\reg}'(\Rn)$ such that $K_{0}(x,y):=\check{\tau}_{\xiy}(x,y)$ is in $\cK_{0}(\URn)$ and is homogeneous of degree $0$ with respect to $y$. 

 Let $Q\in \pdo^{-n}(\Rn)$ have Schwartz kernel $k_{Q}(x,y)=K_{0}(x,x-y)$. Then by Lemma~\ref{lem:Traces.criterion-logarithmic-kernel} 
 we can write $Q$ as a sum of commutators of the form~(\ref{eq:Traces.sum-commutators-K0}). Moreover $Q$ has symbol 
 $q(x,\xi)\sim (K_{0})^{\vee}_{\yxi}(x,\xi)\sim p_{-n}(x,\xi)$, so the operator $\tilde{P}=P-Q$ has symbol 
$\tilde{p}(x,\xi)\sim \sum_{m-j\neq-n}p_{m-j}(x,\xi)$. The first part of the proof then shows that $\tilde{P}$ can be put in the 
form~(\ref{eq:Traces.sum-commutators-cP0}). Incidentally, we see that $P$ can be put in that form.

Finally, let us further assume that $P$ is of regular parity-class. Then $\tilde{P}$ is of regular parity-class and, as its symbol of $\tilde{P}$ of degree 
$-n$ is zero, it follows from the first part of the proof that it can be put in the 
form~(\ref{eq:Traces.sum-commutators-cP0}) where the operators $P_{1},\ldots,P_{n}$ can be chosen to be of regular parity-class. 
The operator $Q$ is of regular parity-class too. 
In fact, as  
$p_{-n}(x,-\xi)=-p_{-n}(x,\xi)$ by~\cite[Lem.~1.3]{Po:NCR} we can choose $\tau(x,\xi)$ to be such that $\tau(x,-\xi)=-\tau(x,\xi)$. Then 
\begin{equation}
    K_{0}(x,-y)=[\tau(x,-\xi)]^{\vee}_{\xiy}=-\check{\tau}_{\xiy}(x,y)=-K_{0}(x,y).
\end{equation}
 Therefore, by Lemma~\ref{lem:Traces.criterion-logarithmic-kernel} 
we can put $Q$ in the form~(\ref{eq:Traces.sum-commutators-K0}) 
with $P_{1},\ldots,P_{n}$ of regular parity-class. Since $P=\tilde{P}+Q$ it follows from all this that if $P$ is of regular parity-class, then the operators 
$P_{1},\ldots,P_{n}$ can be chosen to be of regular parity-class. 
\end{proof}

In the sequel we shall use the notation $\Psi_{c}$ to denote classes of \psidos\ with a compactly supported Schwartz kernel, e.g.,
$\pdoc^{\Z}(\Rn)$ is the space of integer order \psidos\ on $\Rn$ whose Schwartz kernels have compact supports.

\begin{proposition}\label{prop:Traces.sum-commutators-M1}
    Let $P\in \Psi^{m}(M,\cE)$, $m\in \C$, and assume that either $m$ is not an integer~$\geq -n$ or we have $c_{P}(x)=0$. Then $P$ can be 
    written in the form
\begin{equation}
    P=  [ a_{1}, P_{1}]+\ldots+ [a_{N}, P_{N}] +R,
     \label{eq:Traces.sum-commutators-cP0-M}
\end{equation}    
   where $a_{1},\ldots,a_{N}$ are smooth functions on $M$, the operators $P_{1}, \ldots, P_{N}$ are in $\Psi^{m+1}(M,\cE)$ 
   and $R$ is a smoothing operator. Furthermore, if $P$ is in $\pdozo(M,\cE)$ then $P_{1},\ldots, P_{N}$ can be chosen to be in $\pdozo(M,\cE)$ as 
   well. 
\end{proposition}
\begin{proof}
    Let us first assume that $P$ is a scalar \psido\ on $\R^{n}$ such that its Schwartz kernel has compact support. By Lemma~\ref{lem:Traces.sum-commutators-Rn1} 
    there exist 
    $P_{1},\ldots,P_{n}$ in $\Psi^{m+1}(\R^{n})$ and $R\in \psinf(\R^{n})$ such that 
    \begin{equation}
    P=  [ x_{1}, P_{1}]+\ldots+ [x_{n}, P_{n}] +R.
\end{equation}    

Let $\chi$ and $\psi$ in $C^{\infty}_{c}(\Rn)$ be such that $\psi(x)\psi(y)=1$ near the support of the kernel of $P$, so that we have $\psi P 
\psi=P$, and let $\chi \in C^{\infty}_{c}(\Rn)$ be such that $\chi=1$ near $\supp \psi$. As 
$\psi [x_{j},P_{j}]\psi= x_{j} \chi \psi P_{j}\psi- \psi P\psi x_{j}\chi=[x_{j} \chi ,\psi P_{j}\psi]$ we get
\begin{equation}
    P=\psi P \psi= \sum_{j=1}^{n} \psi [x_{j}, P_{j}]\psi +\psi R \psi= \sum_{j=1}^{n}[\chi x_{j}, \psi 
    P_{j}\psi] + \psi R \psi.  
     \label{eq:Traces.commutators.vanishing-cP}
\end{equation}
This shows that $P$  can be put in the form~(\ref{eq:Traces.sum-commutators-cP0-M}) 
with functions $a_{1},\ldots,a_{n}$ of compact supports and operators $P_{1},\ldots,P_{n}$ and $R$ with 
compactly supported Schwartz kernels. 
In addition, if $P$ is of regular parity-class, then Lemma~\ref{lem:Traces.sum-commutators-Rn1}  insures us that $P_{1},\ldots,P_{n}$ can be 
chosen to be of regular parity-class, so that the operators $\psi P_{1} \psi,\ldots, \psi P_{n}\psi$ are of regular parity-class. 
Furthermore, these results immediately extends to \psidos\ in $\pdoc^{*}(U,\cE)$ where $U\subset M$ is local open chart over which $\cE$ is trivializable. 


Next, let $P\in \Psi^{m}(M,\cE)$ and assume that either $m$ is not an integer~$\geq -n$ or we have $c_{P}(x)=0$. Let
$(\varphi_{i})\subset C^{\infty}(M)$ be a finite partition of unity subordinated to an open covering $(U_{i})$ 
by trivializing local charts. For each index $i$ let $\psi_{i}\in C^{\infty}_{c}(U_{i})$ be such that 
$\psi_{i}=1$ near $\supp \varphi_{i}$. Then there exists $R\in \psinf(M,\cE)$ such that
\begin{equation}
    P= \sum \varphi_{i}P\psi_{i} +R.
     \label{eq:Traces-sum-phiPpsi}
\end{equation}

Each operator $P_{i}:=\varphi_{i}P\psi_{i}$ is contained in $\pdoc^{m}(U_{i},\cE)$. Moreover, if $m$ is not an integer~$\geq -n$ then, 
as $P_{i}$ and $\varphi_{i}P$ agrees up to a smoothing operator, 
we have $c_{P_{i}}(x)=c_{\varphi_{i}P}(x)=\varphi_{i}c_{P}(x)=0$. In any case, the first part of the proof insures us that each operator 
$P_{i}$ can be put in the form~(\ref{eq:Traces.sum-commutators-cP0-M}). Thanks to~(\ref{eq:Traces-sum-phiPpsi}) 
it then follows that $P$ can be put in such a form as well. Furthermore, if $P$ is of regular parity-class then the operators $P_{1},\ldots,P_{m}$ can 
be chosen to be of regular parity-class.
\end{proof}

Throughout the paper we will make use of the following.
\begin{lemma}[\cite{Wo:PhD}, \cite{Gu:RTCAFIO}]\label{lem:TracesCZ.smoothing-operators}
1) Any $R \in \psinf(M,\cE)$ such that $\Tra (R)=0$ is the sum of two commutators in $\psinf(M,\cE)$.\smallskip 

2) Any trace on $\psinf(M,\cE)$ is a constant multiple of the usual trace. 
\end{lemma}

We are now ready to prove Theorem~\ref{thm:NCRCT.unicity-TR}.  
\begin{proof}[Proof of Theorem~\ref{thm:NCRCT.unicity-TR}]
    Let $\tau:\pdocz(M,\cE)\rightarrow \C$ be a linear map vanishing on $[C^{\infty}(M), \pdocz(M,\cE)]$ and $[\psinf(M,\cE),\psinf(M,\cE)]$. Then 
    $\tau$ induces a trace on $\psinf(M,\cE)$, so by Lemma~\ref{lem:TracesCZ.smoothing-operators} 
    there exists a constant $c\in \C$ such that, for any $R\in 
    \psinf(M,\cE)$, we have $\tau(R)=c\Tra(R)$. 
    
    Let $P\in \pdocz(M,\cE)$. Then by Proposition~\ref{prop:Traces.sum-commutators-M1} there exists $R\in \psinf(M,\cE)$ such that
    $P= R \ \bmod [C^{\infty}(M), \pdocz(M,\cE)]$.
  Since $\tau$ vanishes on  $[C^{\infty}(M), \pdocz(M,\cE)]$ we have $\tau(P)=\tau(R)=c\Tra(R)$. Similarly, we have 
   $\TR P=\TR R=\Tra (R)$, so we see that $\tau(P)=c\TR P$. Thus $\tau$ is a constant multiple of $\TR$. 
   
   Finally , along similar lines we can prove that any linear form on $\pdozo(M,\cE)$ vanishing on $[C^{\infty}(M), \pdozo(M,\cE)]$ and
   $[\psinf(M,\cE),\psinf(M,\cE)]$ is a constant multiple of the canonical trace. The 
   proof of Theorem~\ref{thm:NCRCT.unicity-TR} is thus achieved.
\end{proof}

We close this section with the following.

\begin{proposition}
1) For any $P\in \pdocz(M,\cE)$ we have $\TR P=0$ if and only if $P$ is contained in 
$[C^{\infty}(M),\pdocz(M,\cE)]+[\psinf(M,\cE),\psinf(M,\cE)]$.\smallskip
    
2) For any $P\in \pdozo(M,\cE)$ we have $\TR P=0$ if and only if $P$ is contained in 
$[C^{\infty}(M),\pdozo(M,\cE)]+[\psinf(M,\cE),\psinf(M,\cE)]$.
\end{proposition}
\begin{proof}
We will only prove the first part, since the second part can be proved along similar lines. 
Let $R_{0}\in \psinf(M,\cE)$ be such that $\Tra R_{0}\neq 0$ (e.g.~$R_{0}=e^{-\Delta_{\cE}}$ where 
$\Delta_{\cE}$ is a Laplace type operator acting on the sections of 
$\cE$). Since by Lemma~\ref{lem:TracesCZ.smoothing-operators} 
the commutator space of $\psinf(M,\cE)$ agrees with the null space of $\Tra_{|\psinf(M,\cE)}$, for any $R\in \psinf(M,\cE)$ we have 
\begin{equation}
    R=\frac{\Tra R}{\Tra R_{0}}R_{0} \quad \bmod [\psinf(M,\cE),\psinf(M,\cE)].
     \label{eq:TracesCZ.smoothing-commutators}
\end{equation}

Let $P\in \pdocz(M,\cE)$. By Proposition~\ref{prop:Traces.sum-commutators-M1} 
there exists $R\in \psinf(M,\cE)$ such that $P=R \ \bmod [C^{\infty}(M),\pdocz(M,\cE)]$. 
Then we have $\TR P=\TR R=\Tra R$, so using~(\ref{eq:TracesCZ.smoothing-commutators}) we get
\begin{equation}
    P=\frac{\TR P}{\Tra R_{0}}R_{0} \ \bmod [C^{\infty}(M),\pdocz(M,\cE)] +[\psinf(M,\cE),\psinf(M,\cE)].
\end{equation}
It then follows that $P$ belongs to $[C^{\infty}(M),\pdocz(M,\cE)]+[\psinf(M,\cE),\psinf(M,\cE)]$ if and only if $\TR P$ vanishes. Hence the result.
\end{proof}

\section{Uniqueness of the noncommutative residue}
\label{sec:TracesZ}
In this section we shall prove Theorem~\ref{thm:NCRCT.unicity-NCR} concerning the uniqueness of the noncommutative residue. First, we have: 

\begin{lemma}\label{lem:TracesZ.smoothing-operators}
Any operator $R\in \psinf(\Rn)$ can be written in the form
\begin{equation}
    R=  [ x_{1}, P_{1}]+\ldots+ [x_{n}, P_{n}], 
     \label{eq:Traces.sum-commutators-R}
\end{equation}    
with $P_{1}, \ldots, P_{n}$ in $\Psi^{-n+1}(\Rn)$.
\end{lemma}
\begin{proof}
 Let $k_{R}(x,y)$ be the Schwartz kernel of $R$. Since $k_{R}(x,y)$ is smooth we have 
 \begin{equation}
     k_{R}(x,y)=k_{R}(x,x)+(x_{1}-y_{1})k_{R_{1}}(x,y)+\ldots +(x_{n}-y_{n})k_{R_{n}}(x,y),
      \label{eq:Traces.Taylor}
 \end{equation}
for some smooth functions $k_{R_{1}}(x,y),\ldots,k_{R_{n}}(x,y)$. For $j=1,\ldots,n$ let $R_{j}$ be the smoothing operator with kernel $k_{R_{j}}(x,y)$ and 
let $Q$ be the operator with kernel $k_{Q}(x,y)=k_{R}(x,x)$. Then by~(\ref{eq:Traces.Taylor}) 
we have  $R=Q+\sum_{j=1}^{n}[x_{j},R_{j}]$. 

To complete the proof it remains to show that $Q$ can be written as a sum of commutators of the form~(\ref{eq:Traces.sum-commutators-R}). 
Observe that the Schwartz kernel of $Q$ can be written as $k_{Q}(x,y)=K_{0}(x,x-y)$ with $K_{0}(x,y)=k_{R}(x,x)$. Obviously 
$K_{0}(x,y)$ belongs to  $\cK_{0}(\Rn\times \Rn)$ 
and is homogeneous of degree $0$ with respect to $y$, so it follows from Lemma~\ref{lem:Traces.criterion-logarithmic-kernel} 
that $Q$ can be written as a sum of commutators of the form~(\ref{eq:Traces.sum-commutators-R}).  The proof is thus achieved. 
\end{proof}

Using the previous lemma we shall prove:
\begin{proposition}\label{prop:TracesZ.commutators-Psinf}
    Any $R\in \Psi^{-\infty}(M,\cE)$ can be written in the form 
    \begin{equation}
        P=[a_{1},P_{1}]+\ldots+ [a_{N},P_{N}]+[R_{1},R_{2}]+[R_{3},R_{4}],
        \label{eq:TracesZ.commutators-cP0}
    \end{equation}
   where the functions $a_{j}$ are in $C^{\infty}(M)$, the operators $P_{j}$ are in $\Psi^{-n+1}(M,\cE)$ and the operators $R_{j}$ are 
   in $\psinf(M,\cE)$. 
\end{proposition}
\begin{proof}
    If $R\in \psinfc(\Rn)$ then Lemma~\ref{lem:TracesZ.smoothing-operators} 
    tells us that $R=\sum_{j=1}^{n}[x_{j},P_{j}]$ with $P_{1},\ldots,P_{n}$ in $\Psi^{-n+1}(\Rn)$. Let 
    $\psi\in C^{\infty}_{c}(\Rn)$ be such that $\psi(x)\psi(y)=1$ near the support of the kernel of $R$ 
    and let $\chi \in C^{\infty}_{c}(\Rn)$ be such that $\chi=1$ near $\supp \psi$. As in~(\ref{eq:Traces.commutators.vanishing-cP}) we have 
$\psi [x_{j},P_{j}]\psi= x_{j} \chi \psi P_{j}\psi- \psi P\psi x_{j}\chi=[x_{j} \chi ,\psi P_{j}\psi]$, so we get
\begin{equation}
    R=[\chi x_{1}, \psi P_{1}\psi] +\ldots+[\chi x_{n}, \psi P_{n}\psi] .
 \end{equation}
More generally, if $U\subset \R^{n}$ is an open local chart which $\cE$ is trivializable, then any $R\in \psinfc(U,\cE)$ can be 
 written in the form
 \begin{equation}
      R=[a_{1}, P_{1}] +\ldots+[a_{n}, P_{n}],
      \label{eq:TracesZ.sum-commutators-smoothing-U}
 \end{equation}
 with $a_{1},\ldots,a_{n}$ in $C^{\infty}_{c}(U)$ and $P_{1},\ldots,P_{n}$ in $\pdoc^{-n+1}(U,\cE)$. 
 
 Now, let $(\varphi_{i})\subset C^{\infty}(M)$ be a partition of unity subordinated to an open covering $(U_{i})$ 
 of $M$ by local trivializing charts. For each index $i$ let $\psi_{i}\in C^{\infty}_{c}(U_{i})$ be such that $\psi_{i}=1$ near $\supp 
 \varphi_{i}$. Then for any $R\in \psinf(M,\cE)$ we have
 \begin{equation}
     R=\sum \varphi_{i}R\psi_{i} +\sum \varphi_{i}R(1-\psi_{i}).
 \end{equation}
 For each index $i$ the operator $ \varphi_{i}R\psi_{i}$ belongs to $\psinfc(U_{i},\cE)$, so 
 by the first part of the proof it can be written as a sum of commutators of the form~(\ref{eq:TracesZ.sum-commutators-smoothing-U}). 
 Moreover, the operator $S:=\sum \varphi_{i}R(1-\psi_{i})$ is smoothing and has a Schwartz kernel that vanishes on the diagonal, so its trace vanishes 
 and by Lemma~\ref{lem:TracesCZ.smoothing-operators} it can be written a the sum of two commutators in $\psinf(M,\cE)$. Hence $R$ can be put in the 
 form~(\ref{eq:TracesZ.commutators-cP0}).  The proof is now complete. 
\end{proof}

\begin{remark}
    Wodzicki~\cite{Wo:PhD} also proved that any smoothing operator is a sum of \psido\ commutators.
\end{remark}

\begin{remark}
    It follows from the proof of 
    Lemma~\ref{lem:TracesZ.smoothing-operators}  that
    the operators $P_{1},\ldots,P_{n}$ in~(\ref{eq:Traces.sum-commutators-R}) can be chosen to be of singular parity-class. 
    Therefore any $R\in \psinf(M,\cE)$ can be written in the form~(\ref{eq:TracesZ.commutators-cP0}) with 
    operators $P_{1},\ldots, P_{N}$ in $\Psi^{-n+1}(M,\cE)$ of singular parity-class. Notice that they cannot choose them to be of regular 
    parity-class, since otherwise this would imply the vanishing of the canonical trace on smoothing operators, which is obviously wrong. 
\end{remark}

Combining Proposition~\ref{prop:Traces.sum-commutators-M1} and Proposition~\ref{prop:TracesZ.commutators-Psinf} we immediately get:

\begin{proposition}\label{prop:TracesZ.commutators-cP0}
 Let $m \in \Z$ and set $\tilde{m}=\max(m,-n)$.  Then any $P\in \Psi^{m}(M,\cE)$ such that $c_{P}(x)=0$ can be written as a sum of commutators of the form 
    \begin{equation}
        P=[a_{1},P_{1}]+\ldots+ [a_{N},P_{N}]+[R_{1},R_{2}]+[R_{3},R_{4}],
    \end{equation}
   where the functions $a_{j}$ are in $C^{\infty}(M)$, the operators $P_{j}$ are in $\Psi^{\tilde{m}+1}(M,\cE)$ and the operators $R_{j}$ are 
   in $\psinf(M,\cE)$. 
\end{proposition}

%

Next, in the sequel we let $\Gamma_{0}\in \Psi^{-n}(\Rn)$ denote the operator with Schwartz kernel $k_{\Gamma_{0}}(x,y)=-\log|x-y|$. In particular we have 
$c_{\Gamma_{0}}(x)=1$. 

\begin{lemma}\label{lem:TracesZ.cGamma0-commutators}
Let $c\in C^{\infty}_{c}(\Rn)$ be such that $\int c(x)dx=0$. Then there exist functions $c_{1},\ldots,c_{n}$ in $C_{c}^{\infty}(\Rn)$ so that we have
\begin{equation}
    c\Gamma_{0}= [\partial_{x_{1}},c_{1}\Gamma_{0}] +\ldots +[\partial_{x_{n}},c_{n}\Gamma_{0}] +Q,
     \label{eq:Traces.sum-commutators-cGamma0}
\end{equation}
for some $Q\in \Psi^{-n}(\Rn)$ such that $c_{Q}(x)=0$.
\end{lemma}
 \begin{proof}
 Since $c(x)$ has compact support and we have $\int c(x)dx=0$ there exist functions $c_{1},\ldots,c_{n}$ in $C^{\infty}_{c}(\Rn)$ such that 
 $c=\sum_{j=1}^{n}\partial_{x_{j}}c_{j}$ (see, e.g., \cite[pp.~24-25]{Po:JFA1}). 
 Set $P= \sum_{j=1}^{n}[\partial_{x_{j}},c_{j}\Gamma_{0}]$. Then $P$ is a \psido\ of order $-n$. Moreover, 
 as by definition $\Gamma_{0}$ has Schwartz kernel $k_{\Gamma_{0}}(x,y)=-\log|x-y|$, the 
  Schwartz kernel of $P$ is
 \begin{multline}
     k_{P}(x,y)= \sum_{j=1}^{n}-(\partial_{x_{j}}-\partial_{y_{j}})(c_{j}(x) \log|x-y|)   \\
     = - \sum_{j=1}^{n}\partial_{x_{j}}c_{j}(x)\log|x-y| - \sum_{j=1}^{n} c_{j}(x)(x_{j}-y_{j})|x-y|^{-2}.
 \end{multline}
 Thus we have $c_{P}(x)=  \sum_{j=1}^{n}\partial_{x_{j}}c_{j}(x)=c(x)=c_{c\Gamma_{0}}(x)$. It then follows that $c\Gamma_{0}= 
 \sum_{j=1}^{n}[\partial_{x_{j}},c_{j}\Gamma_{0}] +Q$ with $Q\in \Psi^{-n}(\Rn)$ such that $c_{Q}(x)=0$.
 \end{proof}

Let $\rho \in C^{\infty}_{c}(\Rn)$ and $\chi\in C^{\infty}_{c}(\Rn)$ be such that $\int \rho(x)dx=1$ and $\chi=1$ near $\supp 
\rho$. Then we have:
 \begin{lemma}\label{lem:TracesZ.commutators-Rn}
  Any $P\in \pdoc^{\Z}(\Rn)$ of order $m\geq -n$ can be written in the form
  \begin{equation}
      P=(\Res P) \rho \Gamma_{0} \chi + [\psi \partial_{x_{1}} \psi, c_{1}\Gamma_{0}\psi] +\ldots + [\psi \partial_{x_{n}} \psi, c_{n}\Gamma_{0}\psi] +Q, 
  \end{equation}
  for some $\psi, c_{1},\ldots,c_{n}$ in $C^{\infty}_{c}(\Rn)$ and some $Q\in \pdoc^{m}(\Rn)$ such that $c_{Q}(x)=0$. 
 \end{lemma}
 \begin{proof}
      Let $P\in \pdoc^{m}(\Rn)$ and set $c(x)=c_{P}(x)-(\Res P)\rho(x)$. Then $c(x)$ is in $C^{\infty}_{c}(\Rn)$ and we have $\int c(x)dx=\int 
      c_{P}(x)dx-\Res P=0$. Therefore, by Lemma~\ref{lem:TracesZ.cGamma0-commutators} 
      there exist functions
$c_{1},\ldots,c_{n}$ such that
\begin{equation}
    c\Gamma_{0}= [\partial_{x_{1}},c_{1}\Gamma_{0}] +\ldots +  [\partial_{x_{n}},c_{n}\Gamma_{0}] +Q,
     \label{eq:Traces.sum-commutators-cGamma02}
\end{equation}
for some $Q\in \Psi^{-n}(\Rn)$ such that $c_{Q}(x)=0$. 

Let $\psi\in C^{\infty}_{c}(\Rn)$ be such that $\psi=1$ near $\supp c \cup \supp c_{1}\cup  
\ldots \cup \supp c_{n}$. Then for $j=1,\ldots,n$ the operator $ \psi [\partial_{x_{j}},c_{j}\Gamma_{0}]\psi$ is equal to
\begin{multline}
\psi (\partial_{x_{j}})\psi c_{j}\Gamma_{0}\psi - 
c_{j}\Gamma_{0} (\partial_{x_{j}})\psi  =    
 [\psi (\partial_{x_{j}})\psi, c_{j}\Gamma_{0}\psi] 
- c_{j}\Gamma_{0}(1-\psi^{2}) (\partial_{x_{j}})\psi.  
     \label{eq:TracesZ.sum-commutators-cGamma02-psi.psi}
\end{multline}
Notice that each operator $c_{j}\Gamma_{0}(1-\psi^{2}) (\partial_{x_{j}})\psi$ is smoothing and has  a compactly supported Schwartz kernel. 
Therefore, by combining~(\ref{eq:Traces.sum-commutators-cGamma02}) and~(\ref{eq:TracesZ.sum-commutators-cGamma02-psi.psi}) we get
 \begin{equation}
     c\Gamma_{0} \psi  =\psi c\Gamma_{0}\psi = 
     \sum_{j=1}^{n} [\psi (\partial_{x_{j}})\psi, c_{j}\Gamma_{0}\psi] +Q,
      \label{eq:TracesZ.sum-commutators-cGamma0psi}
\end{equation}
for some $Q\in \pdoc^{-n}(\Rn)$ such that $c_{Q}(x)=0$. 

Next,  the operator $(\Res P) \rho \Gamma_{0}\chi+ c \Gamma_{0}\psi$ agrees with $(\Res P) \rho \Gamma_{0}+ c 
\Gamma_{0}=c_{P}\Gamma_{0}$ up to a smoothing operator, so its logarithmic singularity is equal to 
$c_{c_{P}\Gamma_{0}}(x)=c_{P}(x)c_{\Gamma_{0}}(x)=c_{P}(x)$. 
Thus 
$P= (\Res P) \rho \Gamma_{0}\chi+ c \Gamma_{0}\psi +Q$, 
for some $Q\in \pdoc^{-m}(\Rn)$ such that $c_{Q}(x)=0$. Together with~(\ref{eq:TracesZ.sum-commutators-cGamma0psi}) this shows that 
\begin{equation}
    P= (\Res P) \rho \Gamma_{0}\chi+ [\psi (\partial_{x_{1}})\psi, c_{1}\Gamma_{0}\psi] +\ldots + [\psi (\partial_{x_{n}})\psi, c_{n}\Gamma_{0}\psi] +Q,
\end{equation}
for some $Q\in \pdoc^{-m}(\Rn)$ such that $c_{Q}(x)=0$. The lemma is thus proved. 
\end{proof}
 
 \begin{proposition}\label{prop:TracesZ.commutators-U}
     Let $U\subset M$ be an open trivializing local chart. Then there exists $P_{0}\in \pdoc(U,\cE)$ so that 
     any $P\in \pdoc^{m}(U,\cE)$, $m \in \Z$, can be written in the form
     \begin{multline}
         P=(\Res P)P_{0} +\sum_{j=1}^{n}[a_{j},P_{j}]+\sum_{j=1}^{n}[L_{j},Q_{j}] +[Q_{n+1},Q_{n+2}]\\ +[R_{1}+R_{2}]+[R_{3}+R_{4}],
          \label{eq:TracesZ.commutators-Rn-Cr}
     \end{multline}
     where the functions $a_{j}$ are in $C^{\infty}_{c}(U)$,  the operators
     $P_{j}$ are in $\pdoc^{\tilde{m}+1}(U,\cE)$  with $\tilde{m}=\sup(m,-n)$, the $L_{j}$ are compactly supported first order differential 
     operators, the $R_{j}$ are in $\psinfc(U,\cE)$ and the operators $Q_{j}$ are in $\pdoc^{-n}(U,\cE)$ and can be chosen to be zero when  $c_{P}(x)=0$.
 \end{proposition}
 \begin{proof}
   First, since $U$ is diffeomorphic to an open subset of $\Rn$ and $\cE$  is trivializable over $U$, 
  we may as well assume that $U=\Rn$ and $\cE$ is trivial. Thus we only have to prove the result for operators in 
  $\pdoc^{\Z}(\Rn,\C^{r})=\pdoc^{\Z}(\Rn)\otimes M_{r}(\C)$. 
   
   Second, it follows from~(\ref{eq:Traces.commutators.vanishing-cP}) 
   and~(\ref{eq:TracesZ.sum-commutators-smoothing-U}) that any $Q\in \pdoc^{m}(\Rn,\C^{r})$ such that $c_{Q}(x)=0$ can be written in the form
 \begin{equation}
       Q=[a_{1},P_{1}]+ \ldots + [a_{n},P_{n}],
      \label{eq:TracesZ.commutators-Rn-Cr-cQ=0}
 \end{equation}
 with $a_{1},\ldots,a_{n}$ in $C^{\infty}_{c}(\Rn)$ and $P_{1},\ldots,P_{n}$ in $\pdoc^{m}(\Rn,\C^{r})$. 
   
   Now, let $P\in \pdoc^{\Z}(\Rn,\C^{r})$ have order~$m\geq-n$. We set $P=(P_{k,l})_{1\leq j,k\leq n}$ and $A=(\Res P_{k,l})_{1\leq j,k\leq n}$. Then  
   applying Lemma~\ref{lem:TracesZ.commutators-Rn} to each operator $P_{j,k}$ shows that there exist compactly supported first order differential 
   operators $L_{1},\ldots,L_{n}$ and operators $Q_{1},\ldots,Q_{n}$ in $\pdoc^{-n}(\Rn,\C^{r})$ such that
\begin{equation}
      P=(\rho \Gamma_{0} \chi)\otimes A +[L_{1}, Q_{1}] +\ldots + [L_{1}, Q_{n}] + Q, 
\end{equation}
  for some $Q\in \pdoc^{m}(\Rn,\C^{r})$ such that $c_{Q}(x)=0$. 
   
   As $\tr A=\sum \Res 
   P_{k,k}=\int \tr c_{P}(x)dx=\Res P$ the matrix $A-\frac{1}{n}(\Res P) I_{n}$ has a zero trace, hence is a commutator
   (see~\cite{Sh:UKM}, \cite{AM:MTZ}). Thus $A=\frac{1}{n}(\Res P) I_{n}+ [A_{1},A_{2}]$, $A_{j}\in M_{r}(\C)$. 
  Set $P_{0}=(\rho \Gamma_{0} \chi)\otimes ( \frac{1}{n}I_{n})$ and $Q_{n+j}=(\rho \Gamma_{0} \chi)\otimes A_{j}$. Then
   \begin{equation}
       P=(\Res P)P_{0}+ [L_{1}, Q_{1}] +\ldots + [L_{1}, Q_{n}] + [Q_{n+1},Q_{n+2}]+Q. 
    \end{equation}
   Combining this with~(\ref{eq:TracesZ.commutators-Rn-Cr-cQ=0}) then shows that $P$ can be put in the form~(\ref{eq:TracesZ.commutators-Rn-Cr}).
 \end{proof}
 
 We are now in position to prove Theorem~\ref{thm:NCRCT.unicity-NCR}. 
\begin{proof}[Proof of Theorem~\ref{thm:NCRCT.unicity-NCR}]
Let $\tau$ be a trace on $\pdoz(M,\cE)$, let $U \subset M$ be a local trivializing open chart, and let 
$\tau_{U}$ denote the restriction of $\tau$ to $\pdoc^{\Z}(U,\cE)$. By Proposition~\ref{prop:TracesZ.commutators-U} 
there exists $P_{0}\in \pdoc^{-n}(U,\cE)$ such that for any $P\in 
\pdoc^{\Z}(U,\cE)$ we have $P=(\Res P)P_{0}$ modulo $[\pdoc^{\Z}(U,\cE), \pdoc^{\Z}(U,\cE)]$. Thus, if we set $c_{U}:=\tau( P_{0})$ then we have
\begin{equation}
    \tau(P)=\tau[(\Res P)P_{0}]=c_{U} \Res P \qquad \forall P\in \pdoc^{\Z}(U,\cE). 
\end{equation}

Let $\Lambda$ be the set of points $x\in M$ near which there is a trivializing open local chart $V$ such that $c_{V}=c_{U}$. This is a 
non-empty open subset of $M$. Let us show that $\Lambda$ is closed as well. Let $x\in \overline{\Lambda}$ and let $V\subset M$ be an open trivializing local 
chart near $x$. Let $y\in \Lambda \cap V$ and let $W$ be a trivializing open local chart near $y$ such that $c_{W}=c_{U}$. As we always can 
find an operator in $\pdoc^{\Z}(V\cup W,\cE)$ such that $\Res P\neq 0$, we must have $c_{V}=c_{V\cup W}=c_{W}=c_{U}$. Thus $x$ belongs to $\Lambda$. 
Hence $\Lambda$ is closed. Since $M$ is connected it follows that $\Lambda$ agrees with $M$, so there exists 
$c\in \C$ such that, for any  open trivializing  local chart $U\subset M$, we have 
\begin{equation}
    \tau(P)=c\Res P \qquad \forall P \in \pdoc^{\Z}(U,\cE).
     \label{eq:Traces.mutiple-local}
\end{equation}

Now, let $(\varphi_{i})$ be a finite partition of the unity subordinated to an open covering $(U_{i})$  of $M$ by local trivializing charts. 
For each index $i$ let $\psi_{i}\in C^{\infty}_{c}(U_{i})$ be such that 
$\psi_{i}=1$ near $\supp \varphi_{i}$. Then any $P \in 
\pdoz(M,\cE)$ can be written as  $P=\sum \varphi_{i} P\psi_{i}+R$, where $R$ is a smoothing operator. By Proposition~\ref{prop:TracesZ.commutators-cP0} 
the operator $R$ is a sum 
of commutators in $\pdoz(M,\cE)$ and for each index $i$ the operator $\varphi_{i} P\psi_{i}$ belongs to $\pdoc^{\Z}(U_{i},\cE)$. Therefore, 
from~(\ref{eq:Traces.mutiple-local}) we get 
\begin{equation}
 \tau(P)= \sum \tau(\varphi_{i}P\psi_{i})=\sum c\Res ( \varphi_{i}P\psi_{i})=c\Res (\sum \varphi_{i}P\psi_{i})=c\Res P.
\end{equation}
This proves that $\tau$ is a constant multiple of the noncommutative residue.
\end{proof}

Finally, as a corollary of Theorem~\ref{thm:NCRCT.unicity-NCR} we have: 

\begin{corollary}[\cite{Wo:PhD}, \cite{Gu:RTCAFIO}]\label{cor:TracesZ.commutators}
    Suppose that $M$ is connected. Then an operator $P\in \pdoz(M,\cE)$ belongs to the commutator space $[\pdoz(M,\cE),\pdoz(M,\cE)]$ if an only if 
    $\Res P$ vanishes.
 \end{corollary}
\begin{proof}
  By Theorem~\ref{thm:NCRCT.unicity-NCR} the space of traces on $\pdoz(M,\cE)$ has dimension $1$, or equivalently, the dual space of
  $\pdoz(M,\cE)/[\pdoz(M,\cE),\pdoz(M,\cE)]$ has dimension 1. Hence the commutator space $[\pdoz(M,\cE),\pdoz(M,\cE)]$ has codimension 1. 
  Let $P_{0}\in \pdoz(M,\cE)$ be such that $\Res P_{0}\neq 0$. This implies that $P_{0}$ is not in 
  $[\pdoz(M,\cE),\pdoz(M,\cE)]$. As the latter has codimension 1 we see that, for any $P\in \pdoz(M,\cE)$, we have 
  \begin{equation}
      P=\lambda P_{0} \quad \bmod [\pdoz(M,\cE),\pdoz(M,\cE)], 
  \end{equation}
  for some $\lambda \in \C$. Observe that $\Res P= \lambda \Res P_{0}$, so we have 
  \begin{equation}
       P=  \frac{\Res P}{\Res P_{0}}P_{0} \quad 
        \bmod [\pdoz(M,\cE),\pdoz(M,\cE)]. 
  \end{equation}
  It then follows that $P$ is in $[\pdoz(M,\cE),\pdoz(M,\cE)]$ if and only if $\Res P$ vanishes. 
\end{proof}

\section{Traces on zero'th order \psidos\ ($n\geq 2$)}
\label{sec:TracesZ0}
 The aim of this section is to prove Theorem~\ref{thm:NCRCT.Traces0} about the characterization in 
 dimension~$\geq 2$ of traces on zero'th order \psidos. 
 
 Recall that for any $P\in  \Psi^{0}(M,\cE)$ the zero'th order symbol of $P$ uniquely defines a section $\sigma_{0}(P)\in C^{\infty}(S^{*}M,\End 
 \cE)$, where $S^{*}M=T^{*}M/\R_{+}$ denotes the cosphere bundle of $M$. In addition, if $L$  is a linear form on $C^{\infty}(S^{*}M)$ 
 then its associated leading symbol trace $\tau_{L}$ is the trace on $\Psi^{0}(M,\cE)$ given by
\begin{equation}
    \tau_{L}(P)=L[\tr_{\cE}\sigma_{0}(P)] \qquad \forall P\in \Psi^{0}(M,\cE).
\end{equation}

Next, let $(\varphi_{i})$ be a partition of the unity subordinated to a covering of $M$ by open domains $U_{i}$ 
of local chart maps $\kappa_{i}:U_{i}\rightarrow V_{i}\subset \Rn$ over which there are trivialization maps $\tau_{i}:\cE_{|U_{i}}\rightarrow U_{i}\times 
\C^{r}$. For index $i$ let $\psi_{i}\in C^{\infty}_{c}(U_{i})$ be such that $\psi_{i}=1$ near $\supp \varphi_{i}$. In addition let 
$\chi\in C^{\infty}_{c}(\Rn)$ be such that $\chi(\xi)=1$ near $\xi=0$. 

For $\sigma \in C^{\infty}(S^{*}M,\End \cE)$ we let $P_{\sigma}\in \Psi^{0}(M,\cE)$ be the \psido\ given by
\begin{equation}
    P_{\sigma}=\sum \varphi_{i} [\tau_{i}^{*}\kappa_{i}^{*}p_{i}(x,D)] \psi_{i}, \quad p_{i}(x,\xi)=(1-\chi(\xi))(\kappa_{i*}\tau_{i*}\sigma)(x,|\xi|^{-1}\xi).
    \label{eq:Traces0.Psigma}
\end{equation}
On each chart $U_{i}$ the operator $\varphi_{i} \tau_{i}^{*}\kappa_{i}^{*}p_{i}(x,D) \psi_{i}$ has principal symbol $\varphi_{i}\sigma $, so we have 
$\sigma_{0}(P)=\sum \varphi_{i}\sigma=\sigma$. Furthermore, the following holds.

\begin{lemma}\label{lem:Traces0.sigma-commutators}
    Let $\sigma \in C^{\infty}(S^{*}M,\End \cE)$. Then:\smallskip
    
    1) We have $c_{P_{\sigma}}(x)=0$.\smallskip
    
    2)  There exist $\sigma_{1},\ldots,\sigma_{N}$ in $C^{\infty}(S^{*}M,\End \cE)$ and 
    $Q\in \Psi^{-1}(M,\cE)$ such that 
    \begin{gather}
        \sigma = \frac{1}{r}(\tr_{\cE}\sigma) \op{id}_{\cE} + [\sigma_{1},\sigma_{2}]+\ldots+[\sigma_{N-1},\sigma_{N}],
         \label{eq:Traces0.sigma-commutators-cE}\\
         P_{\sigma}=P_{ \frac{1}{r}(\tr_{\cE}\sigma) \op{id}_{\cE}} +[P_{\sigma_{1}},P_{\sigma_{2}}]+\ldots+[P_{\sigma_{N-1}},P_{\sigma_{N}}]+Q.
         \label{eq:Traces0.Psigma-commutators-cE}
    \end{gather}
 \end{lemma}
\begin{proof}
    1) In~(\ref{eq:Traces0.Psigma}) the symbol $p_{i}(x,\xi)$ has no homogeneous component of degree $-n$, so we have
    $c_{p_{i}(x,D)}(x)=0$. In addition, since on $C^{\infty}_{c}(U_{i})$ the operators
    $\varphi_{i} [\tau_{i}^{*}\kappa_{i}^{*}p_{i}(x,D) ]\psi_{i}$ and $\varphi_{i} [\tau_{i}^{*}\kappa_{i}^{*}p_{i}(x,D)]$ agree up to a smoothing 
    operator, we have  $c_{\varphi_{i} \tau_{i}^{*}\kappa_{i}^{*}p_{i}(x,D) \psi_{i}}(x)=\varphi_{i}(x) c_{\tau_{i}^{*}\kappa_{i}^{*}p_{i}(x,D)}(x) 
    =\varphi_{i}(x)\tau_{i}^{*}\kappa_{i}^{*}c_{p_{i}(x,D)}(x)=0$. 
    Hence we  have $c_{P_{\sigma}}(x)= \sum c_{\varphi_{i} \tau_{i}^{*}\kappa_{i}^{*}p_{i}(x,D) \psi_{i}}(x)=0$.\smallskip 
    
    2) As for $\sigma_{1}$ and $\sigma_{2}$ in $C^{\infty}(S^{*}M,\End \cE)$ the operators $P_{[\sigma_{1},\sigma_{2}]}$ and 
    $[P_{\sigma_{1}},P_{\sigma_{2}}]$ both have
     principal symbol $[\sigma_{1},\sigma_{2}]$, hence agree modulo an operator in $\Psi^{-1}(M,\cE)$,
    we see that (\ref{eq:Traces0.Psigma-commutators-cE}) follows from~(\ref{eq:Traces0.sigma-commutators-cE}). Therefore, we only have to prove the latter.
     
    Next, any matrix with vanishing trace is a commutator. In fact, it can be seen from the proof in~\cite{Ka:OCCHTZ} that this 
    result  extends to the setting of smooth families of matrices. Therefore,  if $U$ is an open subset of $M$ over which $\cE$ is 
    trivializable, then for any $\sigma \in 
    C^{\infty}(S^{*}U,\End \cE)$ there exist sections $\sigma^{(1)}$ and $\sigma^{(2)}$ in $C^{\infty}(S^{*}U,\End \cE)$ such that
    \begin{equation}
        \sigma(x,\xi)= \frac{1}{r}(\tr_{\cE_{x}}\sigma(x,\xi)) \op{id}_{\cE_{x}} 
        +[\sigma^{(1)}(x,\xi),\sigma^{(2)}(x,\xi)].
         \label{eq:Traces0.sigma-commutators-Cr}
    \end{equation}
 
    Now, let $\sigma \in C^{\infty}(S^{*}M,\End \cE)$ and let $(\varphi_{i})\subset C^{\infty}(M)$ be a finite family of smooth functions such that 
    $\sum \varphi_{i}^{2}=1$ and each function $\varphi_{i}$ 
    has a support contained in an open subset $U_{i}\subset M$ over which $\cE$ is trivializable. For each index $i$ there exist $\sigma^{(j)}_{i}\in 
    C^{\infty}(S^{*}U,\End\cE)$, $j=1,2$,  such that on $S^{*}U_{i}$ we can write $\sigma$ in the form~(\ref{eq:Traces0.sigma-commutators-Cr}). 
    Then we have
    \begin{multline}
        \sigma(x,\xi)=\sum \varphi_{i}(x)^{2}\sigma(x,\xi)=  
          \frac{1}{r}(\tr_{\cE_{x}}\sigma(x,\xi)) \op{id}_{\cE_{x}}  +  \\
         +\sum [\varphi_{i}(x) \sigma^{(1)}_{i}(x,\xi),\varphi_{i}(x) \sigma^{(2)}_{i}(x,\xi)].
    \end{multline}
    This shows that $\sigma(x,\xi)$ is of the form~(\ref{eq:Traces0.sigma-commutators-cE}). The proof is thus achieved.
\end{proof}

Next, we shall show that when $n\geq 2$ in Proposition~\ref{prop:TracesZ.commutators-U} we can replace the first order differential operators 
$L_{j}$ by zero'th order \psidos. The key point is the following alternative version of Lemma~\ref{lem:TracesZ.cGamma0-commutators}. 

\begin{lemma}\label{lem:Traces0.cDelta-n}
Assume $n\geq 2$ and let $c\in C^{\infty}_{c}(\Rn)$ be such that $\int c(x)dx=0$. Then there exist functions 
$c_{1},\ldots,c_{n}$ in $C_{c}^{\infty}(\Rn)$ such that 
\begin{equation}
    c(1+\Delta)^{-\frac{n}{2}}= \sum_{j=1}^{n}[\partial_{x_{j}}(1+\Delta)^{-\frac{1}{2}},c_{j}(1+\Delta)^{\frac{1-n}{2}}] +Q,
     \label{eq:Traces.sum-commutators-cDelta-n}
\end{equation}
with $Q\in \Psi^{-n}(\Rn)$ so that $c_{Q}(x)=0$.
\end{lemma}
\begin{proof}
 First, since $c(1+\Delta)^{-\frac{n}{2}}$ is a \psido\ of order $-n$ with principal symbol $c(x)|\xi|^{-n}$ we have 
\begin{equation}
     c_{c(1+\Delta)^{-\frac{n}{2}}}(x)=(2\pi)^{-n}c(x)\int_{S^{n-1}}d^{n-1}\xi=\frac{|S^{n-1}|}{(2\pi)^{n}}c(x). 
\end{equation}
    
Second, as $c(x)$ has compact support and we have $\int c(x)dx=0$ there exist functions $c_{1},\ldots,c_{n}$ in $C^{\infty}_{c}(\Rn)$ such that 
 $c=\sum_{j=1}^{n}\partial_{x_{j}}c_{j}$. Define
\begin{equation}
     P=\sum_{j=1}^{n}[\partial_{x_{j}}(1+\Delta)^{-\frac{1}{2}},c_{j}(1+\Delta)^{\frac{1-n}{2}}]. 
\end{equation}
Then $P$ is a \psido\ of order  $-n$ with principal symbol $p_{-n}(x,\xi)$ is equal to 
 \begin{multline}
     \sum_{j,k=1}^{n}\frac{1}{i} \left[
     \partial_{\xi_{k}}(i\xi_{j}|\xi|^{-1})\partial_{x_{k}}(c_{j}(x)|\xi|^{1-n}) -   \partial_{\xi_{k}}(c_{j}(x)|\xi|^{1-n}) 
     \partial_{x_{k}}(i\xi_{j}|\xi|^{-1}) \right]       \\
     = \sum_{j=1}^{n}\partial_{x_{j}}c_{j}(x)|\xi|^{-n}-\sum_{j,k=1}^{n}\xi_{j}\xi_{k}\partial_{x_{k}}c_{j}(x)|\xi|^{-(n+2)}\\
     = c(x) |\xi|^{-n}-\sum_{j,k=1}^{n}\xi_{j}\xi_{k}\partial_{x_{k}}c_{j}(x)|\xi|^{-(n+2)}.
 \end{multline}
Therefore, from~(\ref{eq:NCR.formula-cP}) we obtain
\begin{equation}
    (2\pi)^{n}c_{P}(x)= c(x) \int_{S^{n-1}}d^{n-1}\xi-\sum_{k=1}^{n} \partial_{x_{k}}c_{j}(x) \int_{S^{n-1}}\xi_{j}\xi_{k}d^{n-1}\xi.
\end{equation}
 
If $k\neq j$ then for parity reasons the integral $\int_{\R^{n}}x_{j}x_{k}e^{-|x|^{2}}dx$ vanishes, but if we integrate it in polar coordinates then we get
\begin{equation}
    \int_{\R^{n}}x_{j}x_{k}e^{-|x|^{2}}dx=(\int_{0}^{\infty}r^{n+2}e^{-r^{2}}dr)(\int_{S^{n-1}}\xi_{j}\xi_{k}d^{n-1}\xi). 
\end{equation}
Hence $\int_{S^{n-1}}\xi_{j}\xi_{k}d^{n-1}\xi=0$ for $k \neq j$. Furthermore, for $k=j$ we have 
\begin{equation}
    \int_{S^{n-1}}\xi_{j}^{2}d^{n-1}\xi=\frac{1}{n}\sum_{l=1}^{n} 
    \int_{S^{n-1}}\xi_{l}^{2}d^{n-1}\xi=\frac{1}{n}\int_{S^{n-1}}d^{n-1}\xi=\frac{|S^{n-1}|}{n}.
\end{equation}
 
Combining all this we see that
 \begin{equation}
      (2\pi)^{n}c_{P}(x)=c(x)|S^{n-1}|-\sum_{j=1}^{n}\partial_{x_{j}}c_{j}(x)\frac{|S^{n-1}|}{n}=\frac{n-1}{n}|S^{n-1}|c(x).
 \end{equation}
 Thus, if we set $Q= c(1+\Delta)^{-\frac{n}{2}}-\frac{n}{n-1}P$, then $c_{Q}(x)=\frac{|S^{n-1}|}{(2\pi)^{n}}c(x)-\frac{n}{n-1}c_{P}(x)=0$. As 
 $c(1+\Delta)^{-\frac{n}{2}}=\frac{n}{n-1}P+Q=\sum_{j=1}^{n}[\partial_{x_{j}}(1+\Delta)^{-\frac{1}{2}},\frac{n}{n-1}c_{j}(1+\Delta)^{\frac{1-n}{2}}] +Q$
 we then see that $ c(1+\Delta)^{-\frac{n}{2}}$ is of the form~(\ref{eq:Traces.sum-commutators-cDelta-n}). The lemma is thus proved. 
\end{proof}

Thanks to Lemma~\ref{lem:Traces0.cDelta-n} we may argue as in the proofs of Lemma~\ref{lem:TracesZ.commutators-Rn} 
and Proposition~\ref{prop:TracesZ.commutators-U} to get: 

 \begin{proposition}\label{prop:Traces0.commutators-U}
     Assume $n\geq 2$ and let $U \subset M$ be a local open chart over which $\cE$ is trivializable. Then there exists  
     $P_{0}\in \pdoc(U,\cE)$ such that 
     any $P\in \pdoc^{m}(U,\cE)$, $m\in\Z$, can be written in the form 
     \begin{equation}
         P=(\Res P)P_{0}+\sum_{j=1}^{2n}[A_{j},P_{j}]+[Q_{1},Q_{2}]+[R_{1},R_{2}]+[R_{3},R_{4}],
     \end{equation}
     where the $A_{j}$ are in $\pdoc^{0}(U,\cE)$, the $P_{j}$ are in $\pdoc^{m+1}(U,\cE)$ , the operators $Q_{1}$ and $Q_{2}$ are 
     in $\pdoc^{-n}(U,\cE)$, and the $R_{j}$ are in $\psinfc(U,\cE)$.
 \end{proposition}

We are now ready to prove Theorem~\ref{thm:NCRCT.Traces0}. 

\begin{proof}[Proof of Theorem~\ref{thm:NCRCT.Traces0}]
    Let $\tau$ be a trace on the algebra $\Psi^{0}(M,\cE)$. Let $U \subset M$ be a local open chart over which $\cE$ is 
    trivializable. By Proposition~\ref{prop:Traces0.commutators-U} there exists 
     $P_{0}\in \pdoc(U,\cE)$ such that for any $P\in \pdoc^{-1}(U,\cE)$ we have
    \begin{equation}
        P=(\Res P)P_{0} \qquad \bmod [\pdoc^{0}(U,\cE), \pdoc^{0}(U,\cE)]
    \end{equation} 
    It follows that $\tau(P)=\tau(P_{0})\Res P$ for all $P\in \pdoc^{-1}(U,\cE)$. 
    As in the proof of Theorem~\ref{thm:NCRCT.unicity-NCR} we then can show that there exists $c \in \C$ such that  
    \begin{equation}
        \tau(P)=c \Res P \qquad \forall P\in \Psi^{-1}(M,\cE).
    \end{equation}
  
    Next, for $P\in \Psi^{0}(M)$ we set $\tilde{\tau}(P)=\tau(P)-c\Res(P)$. This defines a trace on $\Psi^{0}(M,\cE)$ vanishing on 
    $\Psi^{-1}(M,\cE)$. In addition, we let $L$ be the linear form on $C^{\infty}(S^{*}M)$ such that
    \begin{equation}
        L(\sigma)=\tilde{\tau}(P_{ \frac{1}{r}\sigma \op{id}_{\cE}}) \qquad \forall \sigma \in C^{\infty}(S^{*}M).
    \end{equation}
    
    Let $P \in \Psi^{0}(M,\cE)$.  Since $P-P_{\sigma_{0}(P)}$ has 
    order~$\leq -1$, by Lemma~\ref{lem:Traces0.sigma-commutators} 
    there exist $\sigma_{1},\ldots,\sigma_{N}$ in $C^{\infty}(S^{*}M,\End \cE)$ and 
    $Q\in \Psi^{-1}(M,\cE)$ such that 
    \begin{equation}
         P=P_{\frac{1}{r}(\tr_{\cE}\sigma) \op{id}_{\cE}} +[P_{\sigma_{1}},P_{\sigma_{2}}]+\ldots+[P_{\sigma_{N-1}},P_{\sigma_{N}}]+Q.
         \label{eq:Traces0.P-sum-commutators}
    \end{equation}
 Since $\tilde{\tau}$ is a trace on $\Psi^{0}(M,\cE)$ vanishing on 
    $\Psi^{-1}(M,\cE)$ we get
    \begin{equation}
        \tilde{\tau}(P)= \tilde{\tau}(P_{ \frac{1}{r}(\tr_{\cE}\sigma) \op{id}_{\cE}})=\tau_{L}(\tr_{\cE}\sigma).
    \end{equation}
    Hence $\tau=\tilde{\tau}+c\Res =\tau_{L}+c\Res$.
    
    Next, let us show that the above decomposition of $\tau$ is unique. Suppose that we have another decomposition
    $\tau=\tau_{L_{1}}+c_{1}\Res P$ with $L_{1}\in C^{\infty}(S^{*}M)^{*}$ and  
    $c_{1}\in\C$. Let $P_{0}\in \Psi^{-1}(M,\cE)$ be such that $\Res P_{0}\neq 0$. 
    Then as $\tau_{L}$ vanishes
    on $\Psi^{-1}(M,\cE)$ we get $\tau(P_{0})=c\Res P_{0}$. Similarly, we have $\tau(P)=c_{1}\Res P_{1}$, so $c_{1}=c$. 
    
    On the other hand, if $\sigma \in C^{\infty}(S^{*}M)$ then it follows from Lemma~\ref{lem:Traces0.sigma-commutators} 
    that $\Res(P_{ \frac{1}{r}\sigma \op{id}_{\cE}})=0$, so  
    $\tau(P_{ \frac{1}{r}\sigma \op{id}_{\cE}})=\tau_{L}( \tau(P_{ \frac{1}{r}\sigma \op{id}_{\cE}})=L(\sigma)$.
   The same argument also shows that $\tau(P_{ \frac{1}{r}\sigma \op{id}_{\cE}})=L_{1}(\sigma)$, so we see that $L_{1}=L$. This shows that
    $L$ and $c$ are uniquely determined by $\tau$.  The proof is thus achieved.
  \end{proof}

 Finally, as a corollary to Theorem~\ref{thm:NCRCT.Traces0} we get: 

 \begin{corollary}
     Suppose that $M$ is connected and has dimension $\geq 2$. Then for an operator $P\in \Psi^{0}(M,\cE)$ the following are equivalent:\smallskip
     
     (i) $P$ belongs to the commutator space $[\Psi^{0}(M,\cE),\Psi^{0}(M,\cE)]$.\smallskip
     
     (ii) We have $\tr_{\cE}\sigma_{0}(P)(x,\xi)=0$ and $\Res P=0$.
 \end{corollary}
 \begin{proof}
If $P$ belongs to  $[\Psi^{0}(M,\cE),\Psi^{0}(M,\cE)]$  then $\tr_{\cE}\sigma_{0}(P)(x,\xi)$ and $\Res P$ vanish.   
 Moreover, it follows from the arguments of the first part of the proof of Theorem~\ref{thm:NCRCT.Traces0} that any linear form on 
 $\Psi^{-1}(M,\cE)$ vanishing 
 on $\Psi^{-1}(M,\cE)\cap [\Psi^{0}(M,\cE),\Psi^{0}(M,\cE)]$ is a constant multiple of the noncommutative residue. Therefore, as in the 
 proof of Corollary~\ref{cor:TracesZ.commutators} we can show that an operator $Q\in \Psi^{-1}(M,\cE)$ is contained in $[\Psi^{0}(M,\cE),\Psi^{0}(M,\cE)]$ 
 if and only if $\Res Q$ vanishes. 
 
 Now, let $P\in \Psi^{0}(M,\cE)$ be such that $\tr_{\cE}\sigma_{0}(P)(x,\xi)=0$ and $\Res P=0$. By~(\ref{eq:Traces0.P-sum-commutators}) 
 there exist $\sigma_{1},\ldots,\sigma_{N}$ in $C^{\infty}(S^{*}M,\End \cE)$ and 
    $Q\in \Psi^{-1}(M,\cE)$ such that $P=[P_{\sigma_{1}},P_{\sigma_{2}}]+\ldots+[P_{\sigma_{N-1}},P_{\sigma_{N}}]+Q$.
   Observe that $\Res Q=\Res P=0$, so as $Q$ has order~$\leq -1$ it follows from the discussion above that $Q$ is contained in 
   $[\Psi^{0}(M,\cE),\Psi^{0}(M,\cE)]$. Incidentally $P$ is contained in $[\Psi^{0}(M,\cE),\Psi^{0}(M,\cE)]$ as well. 
 \end{proof}

\section{Traces on zero'th order \psidos\ ($n=1$)}
\label{sec:TracesZ01}
In this section we shall determine all the traces on $\Psi^{0}(M,\cE)$ in dimension 1. The key observation is the following. 

\begin{proposition}\label{prop:Traces01.symbol-1}
    1) For any $P \in \Psi^{0}(M,\cE)$ there exists a unique section $\sigma_{-1}(P)$ in $C^{\infty}(S^{*}M,\End \cE)$ such that, for any local chart 
    map $\kappa:U\rightarrow V$ for $M$ and any trivialization map $\tau:\cE_{|U}\rightarrow U\times \C^{r}$ of $\cE$ over $U$, we have
    \begin{equation}
       [ \kappa_{*}\tau_{*}\sigma_{-1}(P)](x,\xi)=(x,p_{-1}(x,\xi)) \qquad \forall (x,\xi)\in S^{*}V,
         \label{eq:Traces01.symbol-1}
    \end{equation}
    where $p_{-1}(x,\xi)$ is the symbol of degree $-1$ of $P$ in the local coordinates given by $\kappa$ and~$\tau$.\smallskip
    
    2) For any $P_{1}$ and $P_{2}$ in $\Psi^{0}(M,\cE)$ we have
    \begin{equation}
          \sigma_{-1}(P_{1}P_{2})=\sigma_{0}(P_{1})\sigma_{-1}(P_{2})+\sigma_{-1}(P_{1})\sigma_{0}(P_{2}).
           \label{eq:Traces0.product-symbol-1}
     \end{equation}
\end{proposition}
\begin{proof}
 First, let $\phi:U'\rightarrow U$ be a diffeomorphism between open subsets of $\R$ and let $P\in \Psi^{0}(U)$ have symbol $p(x,\xi)\sim \sum 
p_{-j}(x,\xi)$. Then the operator $P'=\phi^{*}P$ is a zero'th order \psido\ on $U'$ whose symbol $p^{\phi}(x,\xi)\sim \sum 
p^{\phi}_{-j}(x,\xi)$ is such that
   \begin{equation}
        p^{\phi}(x,\xi)\sim \sum_{k} a_{k}(x,\xi)(\partial_{\xi}^{k}p)(\phi(x),\phi'(x)^{-1}\xi),
         \label{eq:Traces01.diffeo-symbol}
    \end{equation}
    where $a_{k}(x,\xi)=\frac{1}{k!} \frac{\partial}{\partial y}\left. e^{i\rho_{x}(y)\xi}\right|_{y=x}$ and 
    $\rho_{x}(y)=\phi(y)-\phi(y)-\phi'(x)(x-y)$ (see, e.g.,~\cite{Ho:ALPDO3}). As $\left.d_{y}\rho_{x}\right|_{y=x}=0$ 
    the function $a_{k}(x,\xi)$ is polynomial in $\xi$ 
    of degree~$\leq \frac{k}{2}$, hence $a_{k}(x,\xi)=\sum_{2l\leq k} a_{kl}(x)\xi^{l}$ 
   with $a_{kl}(x)\in C^{\infty}(U')$. Thus, 
    \begin{equation}
        p^{\phi}_{-j}(x,\xi) = \sum_{\substack{j'+k-l=j\\ 2l\leq l}} a_{kl}(x)\xi^{l} 
        (\partial_{\xi}^{k}p_{-j'})(\phi(x),\phi'(x)^{-1}\xi). 
         \label{eq:Traces01.diffeo-homogeneous-symbols}
    \end{equation}

Observe that in dimension 1 the zero'th degree homogeneity implies that we have $p_{0}(x, \xi)=p_{0}(x,\pm 1)$ depending on the sign of $\xi$. In any 
case we have $\partial_{\xi}p_{0}(x,\xi)=0$. Thus for $j=-1$ Eq.~(\ref{eq:Traces01.diffeo-homogeneous-symbols}) reduces to
\begin{equation}
    p^{\phi}_{-1}(x,\xi)=p_{-1}(\phi(x),\phi'(x)^{-1}\xi). 
     \label{eq:Traces01.change-variable-p-1}
\end{equation}

Next, let $Q \in \Psi^{0}(U)$ have symbol $q\sim \sum q_{-j}$ and suppose that $P$ or $Q$ is properly supported. Then $PQ$ belongs to $\Psi^{0}(U)$ 
and has symbol $p\#q \sim \sum \frac{(-i)^{k}}{k!}\partial^{k}_{\xi}p\partial^{k}_{x}q$. Thus its symbol  $(p\#q)_{-1}(x,\xi)$ of degree $-1$ is 
\begin{equation}
       p_{0}q_{-1}+\frac{1}{i}\partial_{\xi}p_{0}\partial_{x}q_{0} +p_{-1}q_{0} 
        =p_{0}q_{-1}+p_{-1}q_{0}.
   \label{eq:Traces0.product-symbol-1-local}
\end{equation}

The formulas~(\ref{eq:Traces01.diffeo-homogeneous-symbols}) and~(\ref{eq:Traces01.change-variable-p-1}) extend \emph{verbatim} to 
vector-valued \psidos\ and matrix-valued symbols. In particular, for $P\in \Psi^{0}(U,\C^{r})$ with symbol $p(x,\xi)\sim \sum p_{-j}(x,\xi)$ 
and for $A$ and $B$ in $C^{\infty}(U,M_{r}(\C))$ the symbol of degree $-1$ of $APB$ is 
\begin{equation}
    (A\#p\# B)_{-1}=A(x)p_{-1}(x,\xi)B(x).
     \label{eq:Traces0.multiplications-functions-symbol-1}
\end{equation}
Together with~(\ref{eq:Traces01.change-variable-p-1}) this shows that for any $P \in \Psi^{0}(M,\cE)$ 
there is a unique section $\sigma_{-1}(P)\in C^{\infty}(S^{*}M,\End \cE)$ satisfying~(\ref{eq:Traces01.symbol-1}). 
Then~(\ref{eq:Traces0.product-symbol-1}) immediately follows 
from~(\ref{eq:Traces0.product-symbol-1-local}). 
\end{proof}

Next, given a  be a linear form on $L$ on $C^{\infty}(S^{*}M)$ we let $\rho_{L}$ denote the linear form on $\Psi^{0}(M,\cE)$ such that
\begin{equation}
    \rho_{L}(P)=L[\tr_{\cE}\sigma_{-1}(P)] \qquad \forall P\in \Psi^{0}(M,\cE).
\end{equation}
If $P_{1}$ and $P_{2}$ are operators in $\Psi^{0}(M,\cE)$, then by~(\ref{eq:Traces0.product-symbol-1}) we have
\begin{multline}
    \rho_{L}(P_{1}P_{2})=L[\tr_{\cE}(\sigma_{0}(P_{1})\sigma_{-1}(P_{2})+\sigma_{-1}(P_{1})\sigma_{0}(P_{2}))] \\ =
    L[\tr_{\cE}(\sigma_{0}(P_{2})\sigma_{-1}(P_{1})+\sigma_{-1}(P_{2})\sigma_{0}(P_{1}))] = \rho_{L}(P_{2}P_{1}).
\end{multline}
Thus $\rho_{L}$ is a trace on the algebra $\Psi^{0}(M,\cE)$. We shall call such a trace a \emph{subleading symbol trace}.
Notice that the noncommutative residue is such a trace, for we have
\begin{equation}
    \Res P= \int_{S^{*}M}\tr_{\cE}\sigma_{-1}(P)(x,\xi) dxd\xi \qquad \forall P\in \Psi^{0}(M,\cE),
\end{equation}
where $dxd\xi$ is the Liouville form of $S^{*}M=T^{*}M/\R_{+}$.

On the other hand, as in~(\ref{eq:Traces0.Psigma}) we can construct a cross-section $\sigma \rightarrow Q_{\sigma}$ from $C^{\infty}(S^{*}M, \End \cE)$ 
to $\Psi^{-1}(M,\cE)$ such that $\sigma_{-1}(Q_{\sigma})=\sigma$ $\forall \sigma \in C^{\infty}(S^{*}M, \End \cE)$. More precisely, for $\sigma\in 
C^{\infty}(S^{*}M, \End \cE)$ we can define $Q_{\sigma}$ to be
\begin{equation}
    Q_{\sigma}=\sum \varphi_{i} [\tau_{i}^{*}\kappa_{i}^{*}q_{i}(x,D) ]\psi_{i}, \quad 
    q_{i}(x,\xi)=(1-\chi(\xi))|\xi|^{-1}(\kappa_{i*}\tau_{i*}\sigma)(x,\frac{\xi}{|\xi|}),
    \label{eq:Traces01.Qsigma}
\end{equation}
where the notation is the same as in~(\ref{eq:Traces0.Psigma}). Then the following holds. 

\begin{lemma}\label{lem:Traces01.sigma-commutators}
    Let $\sigma \in C^{\infty}(S^{*}M,\End \cE)$.\smallskip
    
    1) We have $\sigma_{-1}(P_{\sigma})=0$.\smallskip 
    
    2) There exist  $\sigma_{1},\ldots,\sigma_{N}$ in $C^{\infty}(S^{*}M,\End \cE)$ and $R_{1}, R_{2}$ in $\Psi^{-2}(M,\cE)$ so that  
    \begin{gather}
          P_{\sigma}=P_{ \frac{1}{r}(\tr_{\cE}\sigma) \op{id}_{\cE}} 
          +[P_{\sigma_{1}},P_{\sigma_{2}}]+\ldots+[P_{\sigma_{N-1}},P_{\sigma_{N}}]+R_{1},
          \label{eq:Traces0.Psigma-commutators}\\
        Q_{\sigma}=Q_{ \frac{1}{r}(\tr_{\cE}\sigma) \op{id}_{\cE}} 
          +[Q_{\sigma_{1}},Q_{\sigma_{2}}]+\ldots+[Q_{\sigma_{N-1}},Q_{\sigma_{N}}]+R_{2}.
          \label{eq:Traces0.Qsigma-commutators}
      \end{gather}
 \end{lemma}
 \begin{proof}
   As  the symbol $p_{i}(x,\xi)$ in~(\ref{eq:Traces0.Psigma}) has no homogeneous component of degree~$-1$, 
   from~(\ref{eq:Traces01.symbol-1}) and~(\ref{eq:Traces0.multiplications-functions-symbol-1}) we get 
   $ \sigma_{-1}[\varphi_{i} \tau_{i}^{*}\kappa_{i}^{*}p_{i}(x,D) 
     \psi_{i}]=\varphi_{i}[\tau_{i}^{*}\kappa_{i}^{*}\sigma_{-1}(p_{i}(x,D))]\psi_{i}=0$.
Hence $\sigma_{-1}(P_{\sigma})=\sum \sigma_{-1}[\varphi_{i} \tau_{i}^{*}\kappa_{i}^{*}p_{i}(x,D) 
     \psi_{i}]=0$.     
    
 Next, by Lemma~\ref{lem:Traces0.sigma-commutators} there exist sections $\sigma_{1},\ldots,\sigma_{N}$ in $C^{\infty}(S^{*}M,\End \cE)$ such that 
 $\sigma = \frac{1}{r}(\tr_{\cE}\sigma) \op{id}_{\cE} + \sum [\sigma_{j},\sigma_{j+1}]$. Then $Q_{\sigma}$ 
 and $Q_{ \frac{1}{r}(\tr_{\cE}\sigma) \op{id}_{\cE}} 
          + \sum [Q_{\sigma_{j}},Q_{\sigma_{j+1}}]$ are \psidos\ of order~$-1$ with same principal symbols, 
          so they agree modulo an operator in 
          $\Psi^{-2}(M,\cE)$. Hence $Q_{\sigma}$ is of the form~(\ref{eq:Traces0.Qsigma-commutators}). 

Finally, by~(\ref{eq:Traces0.Psigma-commutators-cE}) we have $P_{\sigma}=P_{\frac{1}{r}(\tr_{\cE}\sigma) \op{id}_{\cE}} + \sum 
[P_{\sigma_{j}},P_{\sigma_{j+1}}]+R_{1}$ with $R_{1}$ in 
 $\Psi^{-1}(M,\cE)$. By the first part of the lemma and by~(\ref{eq:Traces0.product-symbol-1}) we have
 $\sigma_{-1}(P_{\sigma_{j}}P_{\sigma_{j+1}})=
 \sigma_{0}(P_{\sigma_{j}})\sigma_{-1}(P_{\sigma_{j+1}})+\sigma_{-1}(P_{\sigma_{j}})\sigma_{0}(P_{\sigma_{j+1}})=0$. Thus $R_{1}$ is a 
 linear combination of zero'th order \psidos\ whose symbols of degree $-1$ vanish and so $\sigma_{-1}(R_{1})=0$. Since 
 $R_{1}$ has order $\leq -1$ it follows that $R_{1}$ belongs to $\Psi^{-2}(M,\cE)$, proving Eq.~(\ref{eq:Traces0.Psigma-commutators}).
 \end{proof}

Bearing all this in mind we have:
 
 \begin{theorem}\label{thm:Traces01.main}
     Assume that $\dim M=1$. Then:\smallskip
     
     1) Any trace on $\Psi^{0}(M,\cE)$ can be uniquely written as the sum of a leading symbol trace and of a subleading symbol trace.\smallskip
     
     2) An operator $P\in \Psi^{0}(M,\cE)$ belongs to $[\Psi^{0}(M,\cE),\Psi^{0}(M,\cE)]$ if and only we have 
     $\tr_{\cE}\sigma_{0}(P)=\tr_{\cE}\sigma_{-1}(P)=0$.
 \end{theorem}
 \begin{proof}
   First, if $P\in \Psi^{0}(M,\cE)$ belongs to $[\Psi^{0}(M,\cE),\Psi^{0}(M,\cE)]$ then 
     $\tr_{\cE}\sigma_{0}(P)$ and $\tr_{\cE}\sigma_{-1}(P)$ both vanish. 
     
     Conversely, let $P\in \Psi^{0}(M,\cE)$. Since by Lemma~\ref{lem:Traces01.sigma-commutators} 
     we have $\sigma_{-1}(P_{\sigma_{0}(P)})=0$ we see that the symbols of degree 0 and 
     $-1$ of $P-P_{\sigma_{0}(P)}-Q_{\sigma_{-1}(P)}$ are both zero, hence $P-P_{\sigma_{0}(P)}-Q_{\sigma_{-1}(P)}$
    has order~$\leq -2$. Combining this with the second part of Lemma~\ref{lem:Traces01.sigma-commutators} 
    we then deduce that $P$ can be written in the form
    \begin{equation}
        P=P_{\frac{1}{r}(\tr_{\cE}\sigma_{0}(P)) \op{id}_{\cE}} + Q_{\frac{1}{r}(\tr_{\cE}\sigma_{-1}(P)) \op{id}_{\cE}} +R,
    \end{equation}
     for some $R$ in 
    $\Psi^{-2}(M,\cE)$. Since in dimension 1 Proposition~\ref{prop:TracesZ.commutators-cP0} 
    implies that $\Psi^{-2}(M,\cE)$ is contained in $[\Psi^{0}(M,\cE),\Psi^{0}(M,\cE)]$ 
    we see that
    \begin{equation}
        P=P_{\frac{1}{r}(\tr_{\cE}\sigma_{0}(P)) \op{id}_{\cE}} + Q_{\frac{1}{r}(\tr_{\cE}\sigma_{-1}(P)) \op{id}_{\cE}}
        \quad \bmod [\Psi^{0}(M,\cE),\Psi^{0}(M,\cE)].
        \label{eq:Traces01.P-sum-commutators}
    \end{equation}
    In particular if  $\tr_{\cE}\sigma_{0}(P)=\tr_{\cE}\sigma_{-1}(P)=0$ then $P$ belongs to $[\Psi^{0}(M,\cE),\Psi^{0}(M,\cE)]$. 
    
    Next, let $\tau$ be a trace on $\Psi^{0}(M,\cE)$ and let $L_{1}$ and $L_{2}$ be the linear forms on $C^{\infty}(S^{*}M)$ such that, for any 
    $\sigma \in C^{\infty}(S^{*}M)$, we have
    \begin{equation}
        L_{1}(\sigma)=\tau(P_{\frac{1}{r}\sigma \op{id}_{\cE}}) \quad \text{and} \quad  L_{2}(\sigma)=\tau(Q_{\frac{1}{r}\sigma \op{id}_{\cE}}). 
         \label{eq:Traces0.L1-L2}
    \end{equation}
    Let $P\in \Psi^{0}(M,\cE)$. Then it follows from~(\ref{eq:Traces01.P-sum-commutators}) that $\tau(P)$ is equal to
    \begin{equation}
        \tau[P_{\frac{1}{r}(\tr_{\cE}\sigma_{0}(P)) \op{id}_{\cE}}] + \tau[Q_{\frac{1}{r}(\tr_{\cE}\sigma_{-1}(P)) \op{id}_{\cE}}] 
        =L_{1}(\tr_{\cE}\sigma_{0}(P))+L_{2}(\tr_{\cE}\sigma_{-1}(P)).
    \end{equation}
    Hence $\tau=\tau_{L_{1}}+\rho_{L_{2}}$. 
    
    Assume now that there is another pair $(L_{1}',L_{2}')$ of linear forms on $C^{\infty}(S^{*}M)$ such that $\tau=\tau_{L_{1}'}+\rho_{L_{2}'}$. 
    Let $\sigma \in C^{\infty}(S^{*}M)$. As $\sigma_{-1}(P_{\frac{1}{r}\sigma \op{id}_{\cE}})=0$ have $\rho_{L_{2}'}(P_{\frac{1}{r}\sigma 
    \op{id}_{\cE}})$, and so we get
    $L_{1}(\sigma)=\tau(P_{\frac{1}{r}\sigma \op{id}_{\cE}})=\tau_{L_{1}'}(P_{\frac{1}{r}\sigma \op{id}_{\cE}})=L_{1}'(\sigma)$. 
    Similarly, we have $L_{2}=L'_{2}$, so the decomposition $\tau=\tau_{L_{1}}+\rho_{L_{2}}$ is unique. 
  \end{proof}
 
 Finally, when $\cE$ is the trivial line bundle the condition $\sigma_{0}(P)=\sigma_{-1}(P)=0$ means that $P$ has order~$\leq -2$. Since in dimension 
 1 Proposition~\ref{prop:TracesZ.commutators-cP0} implies that $\Psi^{-2}(M)$ is contained in $[\Psi^{0}(M),\Psi^{0}(M)]$ we obtain:
   
 \begin{corollary}
     When $\dim M=1$ we have $[\Psi^{0}(M),\Psi^{0}(M)]=\Psi^{-2}(M)$.
 \end{corollary}

 \end{document}